\documentclass[xcolor=x11names,reqno,12pt]{amsart}
 \usepackage{amsmath}
 \usepackage{amsthm}
 \usepackage{amssymb}
 \usepackage{latexsym,longtable}
 \usepackage{graphicx}
 \usepackage{multicol}
 \usepackage{mathrsfs}
 \usepackage{young}
 \usepackage[vcentermath,enableskew,stdtext]{youngtab}
 \usepackage[left=1.1in,right=1.1in,top=1.08in,bottom=1.1in]{geometry}
 \usepackage[all]{xy}
    \SelectTips{cm}{10}     
    \everyxy={<2.5em,0em>:} 
 \usepackage{fancyhdr}      
 \usepackage{multicol}
 \usepackage{multirow}
 \usepackage{enumitem}
 \usepackage{bm}
 \usepackage{stmaryrd}
 \usepackage{etex}
 \usepackage{tikz}
 \usepackage{float}
 \usepackage{array}
 \usepackage{colortbl}
 \usepackage{mathtools}
 \restylefloat{figure}
\numberwithin{equation}{section}
 \usetikzlibrary{snakes}
\usepackage{url}
 \usepackage{tabmac}
 \usepackage{ytableau}

\usepackage{amsfonts,epsfig,color}
\makeatletter
\newcommand*{\centerfloat}{%
  \parindent \z@
  \leftskip \z@ \@plus 1fil \@minus \textwidth
  \rightskip\leftskip
  \parfillskip \z@skip}
\makeatother
\newcounter{ctr}
\theoremstyle{plain}
\newtheorem{theorem}{Theorem}[section]
\newtheorem{lemma}[theorem]{Lemma}
\newtheorem{corollary}[theorem]{Corollary}
\newtheorem{proposition}[theorem]{Proposition}
\newtheorem{conjecture}[theorem]{Conjecture}

\theoremstyle{definition}
\newtheorem{definition}[theorem]{Definition}

\newtheorem{remark}[theorem]{Remark}
\newtheorem{example}[theorem]{Example}

\newcommand{\ignore}[1]{}

\newcommand{\A}{{\ensuremath{\mathcal{A}}}}

\newcommand{\bb}{\ensuremath{\mathbf{B}}}

\newcommand{\bD}{\ensuremath{\mathbf{D}}}
\newcommand{\CC}{\ensuremath{\mathbb{C}}}

\newcommand{\End}{\text{\rm End}}

\newcommand{\HH}{\ensuremath{H}}
\newcommand{\HHH}{\bb}

\newcommand{\Mod}{\ensuremath{\mathbf{Mod}}}

\newcommand{\p}{\ensuremath{\mathfrak{p}}}

\newcommand{\QQ}{\ensuremath{\mathbb{Q}}}
\newcommand{\R}{\ensuremath{\mathbf{R}}}

\newcommand{\NN}{\ensuremath{\mathbb{N}}}

\newcommand{\fs}{\ensuremath{\mathfrak{s}}}

\newcommand{\ZZ}{\ensuremath{\mathbb{Z}}}

\newcommand{\be}{\begin{equation}}
\newcommand{\ee}{\end{equation}}

\renewcommand{\SS}{\ensuremath{\mathcal{S}}}
\newcommand{\tsr}{\ensuremath{\otimes}}

\newcommand{\eS}{\widehat{\SS}}

\newcommand{\creading}{\text{\rm colword}}

\newcommand{\mysquare}{}
\newcommand{\mysquareb}{\, }

\DeclareMathOperator{\inside}{inside}
\DeclareMathOperator{\outside}{outside}

\DeclareMathOperator{\diagram}{diagram}
\DeclareMathOperator{\upp}{up}
\DeclareMathOperator{\down}{down}
\DeclareMathOperator{\chainup}{top}

\DeclareMathOperator{\chaindown}{bot}
\DeclareMathOperator{\downpath}{downpath}

\DeclareMathOperator{\bpath}{path}

\DeclareMathOperator{\Par}{Par}
\DeclareMathOperator{\tpar}{fPar}
\DeclareMathOperator{\spin}{spin}

\DeclareMathOperator{\Span}{span}

\DeclareMathOperator{\Gr}{Gr}
\DeclareMathOperator{\Fl}{Fl}

\DeclareMathOperator{\crop}{crop}
\DeclareMathOperator{\nr}{\overline{\mathbf{r}}}
\DeclareMathOperator{\style}{band}

\newcommand{\dd}{{\mathbf d}}

\newcommand{\crc}[1]{#1\star}

\newcommand{\Inv}{\ensuremath{{\text{\rm Inv}}}}

\newcommand{\myhat}[1]{#1}  
\newcommand{\sone}{\ensuremath{\bar{s}}}

\newcommand{\SMT}{\text{\rm SMT}}

\newcommand{\SSYT}{\text{\rm SSYT}}

\newcommand{\CTAB}{\text{\rm FTAB}}

\newcommand{\HRI}{\ensuremath{H(\Psi,\gamma)}}

\newlength{\mycellsize}
\newcommand\mytbl[1]{
\vcenter{
\let\\=\cr
\baselineskip=-16000pt \lineskiplimit=16000pt \lineskip=0pt
\halign{&\mytblcell{##}\cr#1\crcr}}}

\newcommand{\mytblcell}[1]{{%
\def \arg{#1}\def \void{}%
\ifx \void \arg
\vbox to \mycellsize{\vfil \hrule width \mycellsize height 0pt}%
\else \unitlength=\mycellsize
\begin{picture}(1,1)
\put(0,0){\makebox(1,1){$#1\vphantom{\crc{#1}}$}}
\put(0,0){\line(1,0){1}}
\put(0,1){\line(1,0){1}}
\put(0,0){\line(0,1){1}}
\put(1,0){\line(0,1){1}}
\end{picture}%
\fi}}

\newcommand*\encircle[1]{\tikz[baseline=(char.base)]{
    \node[shape=circle,draw,inner sep=.74pt] (char) {\ensuremath{#1}};}}

\newlength{\cellsize}
\newcommand\mytableau[1]{
\vcenter{
\let\\=\cr
\baselineskip=-16000pt \lineskiplimit=16000pt \lineskip=0pt
\halign{&\mytableaucell{##}\cr#1\crcr}}}

\newcommand{\mytableaucell}[1]{{%
\def \arg{#1}\def \void{}%
\ifx \void \arg
\vbox to \cellsize{\vfil \hrule width \cellsize height 0pt}%
\else \unitlength=\cellsize
\begin{picture}(1,1)
\put(0,0){\makebox(1,1){$#1\vphantom{\crc{#1}}$}}
\put(0,0){\line(1,0){1}}
\put(0,1){\line(1,0){1}}
\put(0,0){\line(0,1){1}}
\put(1,0){\line(0,1){1}}
\end{picture}%
\fi}}

\newcommand\boldtableau[1]{
\vcenter{
\let\\=\cr
\baselineskip=-16000pt \lineskiplimit=16000pt \lineskip=0pt
\halign{&\boldtableaucell{##}\cr#1\crcr}}}

\newcommand{\boldtableaucell}[1]{{%
\def \arg{#1}\def \void{}%
\ifx \void \arg
\vbox to \cellsize{\vfil \hrule width \cellsize height 0pt}%
\else \unitlength=\cellsize
\begin{picture}(1,1)
\put(0,0){\makebox(1,1){$\mathbf{#1\vphantom{\crc{#1}}}$}}
\put(0,0){\line(1,0){1}}
\put(0,1){\line(1,0){1}}
\put(0,0){\line(0,1){1}}
\put(1,0){\line(0,1){1}}
\end{picture}%
\fi}}

\title{$k$-Schur expansions of Catalan functions}


\begin{document}

\author{Jonah Blasiak}
\address{Department of Mathematics, Drexel University, Philadelphia, PA 19104}
\email{jblasiak@gmail.com}

\author{Jennifer Morse}
\address{Department of Math, University of Virginia, Charlottesville, VA 22904}
\email{morsej@virginia.edu}

\author{Anna Pun}
\address{Department of Mathematics, Drexel University, Philadelphia, PA 19104}
\email{annapunying@gmail.com}

\author{Daniel Summers}
\address{Department of Mathematics, Drexel University, Philadelphia, PA 19104}
\email{danielsummers72@gmail.com}

\thanks{Authors were supported by NSF Grants DMS-1600391 (J.~B.)
and DMS-1833333  (J.~M.). }

\begin{abstract}
We make a broad conjecture about the $k$-Schur positivity of Catalan functions,
symmetric functions
which generalize the (parabolic) Hall-Littlewood polynomials.
We resolve the conjecture with positive combinatorial formulas in cases which
address the $k$-Schur expansion of
(1) Hall-Littlewood polynomials,
proving the  $q=0$ case of the strengthened Macdonald positivity conjecture from \cite{LLM};
(2) the product of a Schur function and a $k$-Schur function when the indexing partitions
concatenate to a partition, describing a class of Gromov-Witten invariants
for the quantum cohomology of complete flag varieties;
(3) $k$-split polynomials, solving a substantial special case of
a problem of Broer and Shimozono-Weyman on parabolic Hall-Littlewood polynomials~\cite{SW}.
In addition,
we prove the conjecture that the $k$-Schur functions
 defined via $k$-split polynomials \cite{LMksplit}
agree with those defined in terms of strong tableaux~\cite{LLMSMemoirs1}.
\end{abstract}


\maketitle

\section{Introduction}

Catalan functions are elements of the ring
$\Lambda = \ZZ[t][h_1,h_2,\dots]$ of symmetric functions in infinitely many variables
$\mathbf{x} = (x_1, x_2, \dots)$,
where  $h_d = h_d(\mathbf{x})=\sum_{i_1\leq\cdots \leq i_d} x_{i_1}\cdots x_{i_d}$.
Studied in full generality first by Chen-Haiman \cite{ChenThesis} and Panyushev \cite{Panyushev},
Catalan functions are $GL_\ell$-equivariant Euler characteristics of vector bundles
on the flag variety.

They
can be defined by the following Demazure-operator formula.
Fix $\ell \in \ZZ_{\ge 0}$. 
A \emph{root ideal} is an upper order ideal
of the poset  $\Delta^+ = \Delta^+_\ell := \{(i,j) \mid 1 \le i < j \le \ell \big\}$
with partial order given by $(a,b) \leq (c,d)$ when $a\geq c$ and $b\leq d$.
An {\it indexed root ideal of length $\ell$} is
a pair $(\Psi, \gamma)$ consisting of a root ideal $\Psi\subset\Delta_\ell^+$ and a weight $\gamma\in\ZZ^\ell$.
Schur functions can be defined for any $\gamma\in\ZZ^\ell$~by
\begin{align}
\label{ed s gamma}
s_\gamma = s_\gamma(\mathbf{x}) = \det( h_{\gamma_i + j-i}(\mathbf{x}) )_{1 \le i, j \le \ell}
\ \in\, \Lambda,
\end{align}
where by convention $h_0(\mathbf{x}) = 1$  and $h_d(\mathbf{x}) = 0$ for  $d < 0$.

\begin{definition}
\label{d HH gamma Psi}
The \emph{Catalan function} associated to
an indexed root ideal $(\Psi, \gamma)$ of length~$\ell$ is
\vspace{-2.4mm}
\begin{align}
\label{e d HH gamma Psi}
H(\Psi;\gamma)(\mathbf{x};t) := \prod_{(i,j) \in \Psi} \big(1-tR_{i j}\big)^{-1} s_\gamma(\mathbf{x}) \ \in \, \Lambda,
\end{align}
where the raising operator $R_{i j}$ acts on the subscripts of the $s_\gamma$ by $R_{i j} s_\gamma = s_{\gamma + \epsilon_i - \epsilon_j}$ and
$\epsilon_i \in  \ZZ^\ell$ denotes the weight with a 1 in position  $i$ and 0's elsewhere;
for a discussion of raising operators, including a more precise version of \eqref{e d HH gamma Psi}, see \cite[\S4.3]{BMPS}.
By convention,
$H(\varnothing ; \emptyset) := 1$
when  $\ell = 0$,
where $\varnothing$ denotes the empty set and $\emptyset$ denotes the weight/partition of length 0.
\end{definition}


Twenty years ago, symmetric functions known as $k$-Schur functions were discovered in the study of the Macdonald positivity conjecture \cite{LLM} and were subsequently connected
to affine Schubert calculus \cite{LMktableau, LMquantum, LamSchubert}.
Many conjecturally equivalent candidates for $k\text{-Schur}$ functions have been proposed in the years since their discovery.
Among them is
the following subclass of Catalan functions, which we recently connected to other $k$-Schur candidates \cite{BMPS}, proving a conjecture of
Chen-Haiman \cite{ChenThesis} (see \S\ref{ss unifying}).

Fix $k \in \ZZ_{\ge 1}$ and $\ell \in \ZZ_{\ge 0}$.
Write $\Par^k_\ell = \{(\mu_1, \dots, \mu_\ell) \in \ZZ^\ell :  k \ge \mu_1 \ge \cdots \ge  \mu_\ell \ge 0 \}$
for the set of partitions contained in the $\ell \times k$-rectangle and
$\Par^k = \bigsqcup_{m \ge 0} \! \big\{(\mu_1,  \dots, \mu_m) \in \ZZ^m \mid $ $k \ge \mu_1 \ge \mu_2 \ge  \cdots \ge \mu_m > 0 \big\}$ for the set of  $k$-bounded partitions.
For any partition  $\mu$ (elements of  $\Par^k_\ell$ and  $\Par^k$ included), the \emph{length} $\ell(\mu)$ is
the number of nonzero parts of~$\mu$.

\begin{definition}
\label{d0 catalan kschur}
For  $\mu \in \Par^k_\ell$, define the root ideal
\begin{align}
\Delta^k(\mu) = \{(i,j) \in \Delta^+_\ell \mid  k-\mu_i + i < j\},
\end{align}
and the \emph{$k$-Schur Catalan function}
\begin{align}
\label{d catalan kschur}
&\fs^{(k)}_\mu(\mathbf{x};t) := H(\Delta^k(\mu);\mu)(\mathbf{x};t)=
\prod_{i=1}^\ell\,\prod_{j=k+1-\mu_i+i}^\ell \big(1-tR_{i j}\big)^{-1} s_\mu(\mathbf{x})
\,.
\end{align}
\end{definition}

\subsection{A $k$-Schur positivity conjecture}
\label{ss kSchur positivity conjecture}
A central theme in the investigations of Catalan functions
\cite{SW, ChenThesis,Panyushev}
can be articulated as Schur positivity conjectures.  We have discovered that a conjecture
addressing $k$-Schur positivity
encompasses
the earlier conjectures as well as questions
concerning Gromov-Witten invariants and Macdonald polynomials.

Let  $H_\mu(\mathbf x;t)$ denote the modified Hall-Littlewood polynomial indexed by the partition~$\mu$, which is equal to the Catalan function $H(\Delta^+;\mu)(\mathbf{x};t)$.

\begin{proposition}
The $k$-Schur Catalan functions $\{\fs^{(k)}_\mu\}_{\mu \in \Par^k}$ form a $\ZZ[t]$-basis for
\begin{align}
\Lambda^k :=
\Span_{\ZZ[t]} \! \big\{H_{\mu}(\mathbf{x};t) \mid \mu \in \Par^k \big\}\, \subset \Lambda.
\end{align}
\end{proposition}

\begin{proof}
Note that  $\Lambda$ is a free $\ZZ[t]$-module with bases  $\{h_\mu\}$, $\{s_\mu\}$,
$\{H_\mu\}$  ($\mu$ ranges over partitions) since
 $\{h_\mu\}$ and $\{s_\mu\}$ are related by a unitriangular change of basis, as
are $\{H_\mu\}$ \linebreak[4] and $\{s_\mu\}$.
Thus  $\Lambda^k$ is a free  $\ZZ[t]$-module.
Since $\{\fs^{(k)}_\mu\}_{\mu \in \Par^k}$ and $\{H_\mu\}_{\mu \in \Par^k}$ are related by a unitriangular
change of basis by the proof of \cite[Theorem 2.8]{BMPS}, the result follows.
%
%
\end{proof}

Given an indexed root ideal $(\Psi, \mu)$ of length $\ell$, for each $i \in [\ell]$ define
\begin{align}
\nr(\Psi)_i &:= \big|\big\{(i,j) \in \Delta^+ \setminus \Psi : j \in [\ell] \big\}\big|,\\
\style(\Psi, \mu)_i &:= \nr(\Psi)_i + \mu_i.
\end{align}

\begin{proposition}
\label{p style k space}
If $(\Psi, \mu)$  is an indexed root ideal with
$\style(\Psi,\mu)_i \le k$ for all $i$,
then $H(\Psi; \mu) \in \Lambda^k$.
\end{proposition}
\begin{proof}
It was established in the proof of Theorem 2.8 on page 17 of \cite{BMPS}
that $H(\Psi; \mu) \in \Lambda^k$
when  $\mu \in \Par^k_\ell$ and $\Psi = \Delta^k(\mu)$,
but all that was used about $(\Psi, \mu)$ is $\style(\Psi,\mu)_i \le k$ for all $i$, so the proof applies in the present setting unchanged.
\end{proof}

\begin{conjecture}
\label{cj kschur pos}
If $(\Psi, \mu)$  is an indexed root ideal with $\mu \in \Par^k$ and
$\style(\Psi,\mu)_i \le k$ for all $i$,
then the Catalan function $H(\Psi;\mu)$ is $k$-Schur positive.
That is,
\begin{equation}
\label{eq:catalan_in_kschur}
H(\Psi;\mu)(\mathbf x;t)=\sum_{\lambda\in\Par^k} K^{\Psi,k}_{\lambda, \mu}(t)\,\fs^{(k)}_\lambda(\mathbf{x};t)
\quad \text{with} \ \, K^{\Psi,k}_{\lambda,\mu}(t)\in \NN[t]\,.
\end{equation}
\end{conjecture}

We call the  $K^{\Psi,k}_{\lambda,\mu}(t)$ {\it $k$-Catalan-Kostka coefficients}.



\begin{remark}
\label{r kschur conjecture}
The branching property \cite[Theorem 2.6]{BMPS} --- $\fs_\nu^{(k)}$ is  $k+1$-Schur positive
for all  $\nu \in \Par^k$ --- tells us that
 for fixed  $(\Psi, \mu)$ the conjectured  $k$-Schur positivity of $H(\Psi ; \mu)$
is strongest for $k = \max_i\{\, \style(\Psi, \mu)_i\, \}$.  Also,  $k$-Schur positivity implies
Schur positivity  since a $k$-Schur function reduces to a Schur function for large  $k$.
\end{remark}

Conjecture~\ref{cj kschur pos} is a broad conjecture which encompasses several open positivity problems
in algebraic combinatorics.
We elaborate below and preview our
main results which resolve three natural
special cases of this conjecture.

\textbf{Strengthened Macdonald positivity.}
Lapointe, Lascoux, and Morse \cite{LLM}
constructed a
family of symmetric functions---now one of many conjecturally equivalent  $k$-Schur
candidates---and conjectured
 (i) they form a basis for $\Lambda^k$,
(ii) they are Schur positive, and
(iii) for  $\mu \in \Par^k$, the expansion of the modified Macdonald polynomial $H_\mu(\mathbf {x};q,t)$ in this basis has coefficients
in $\NN[q,t]$.
This factors the problem of Schur
expanding Macdonald polynomials 
into the two positivity problems (ii) and (iii).
In \cite{BMPS}, we established that
 $\{\fs^{(k)}_\mu\}_{\mu \in \Par^k}$ is a Schur positive basis for $\Lambda^k$, i.e.,  it satisfies (i)--(ii).
Here we resolve the  $q=0$ specialization of (iii) in the strongest possible terms by
giving  a positive combinatorial formula for the $k$-Schur expansion of
$H_\mu(\mathbf{x}; 0,t) = H_\mu(\mathbf {x};t)$ for all $\mu \in \Par^k$ (Theorem~\ref{t HL to k schur intro});
since $H_\mu(\mathbf{x}; 0,t) = H(\Delta^+; \mu)$, this resolves Conjecture \ref{cj kschur pos}
in the case $\Psi=\Delta^+$.

\textbf{Gromov-Witten invariants.}
The $k$-Schur expansion coefficients in products of two \linebreak[3] $k$-Schur functions at $t=1$
agree (see Section \ref{s gromov witten})
with the 3-point Gromov-Witten invariants of genus 0 which, informally, count equivalence classes of
certain rational curves in the flag variety $\Fl_{k+1}$.
They are structure constants of the quantum cohomology ring of $\Fl_{k+1}$ and specialize to
Schubert polynomial structure constants.  It turns out that  the product of $k$-Schur functions
is realizable as a single Catalan function, as we now explain.

Define a binary operation  $\uplus$ on root ideals as follows: for  $\Psi \subset \Delta^+_r$ and $\Phi \subset \Delta^+_m$,
the root ideal $\Psi \uplus \Phi \subset \Delta^+_{r+m}$ is the result of placing $\Psi$ and  $\Phi$ catty-corner
and including the full  $r \times m$ rectangle of roots
in between;  equivalently, $\Psi\uplus\Phi$ is determined by
\begin{align}
\label{ed root ideal uplus}
\Delta_{r+m}^+ \setminus (\Psi\uplus\Phi) = (\Delta^+_r \setminus \Psi) \sqcup \big\{(i+r,j+r): (i,j)\in \Delta^+_m \setminus \Phi\big\}.
\end{align}

Let $\mu\in\Par^k_r$, $\nu \in \Par^k_m$ and
$\mu\nu = (\mu, \nu)$ denote the concatenation of  $\mu$ and  $\nu$.
We have $H(\Delta^k(\mu)\uplus\Delta^k(\nu) ; \mu\nu)(\mathbf x;1)=
\fs_\mu^{(k)}(\mathbf x;1) \fs_\nu^{(k)}(\mathbf x;1)$ and thus
its $k$-Schur expansion coefficients are the (nonnegative)
Gromov-Witten invariants.
Conjecture \ref{cj kschur pos} predicts a stronger result in the more restricted setting
when $\mu \nu$ is a partition:
the  $k$-Schur expansion coefficients of $H(\Delta^k(\mu)\uplus\Delta^k(\nu); \mu\nu)(\mathbf x;t)$
lie in  $\NN[t]$.
We resolve this with a positive combinatorial formula in the case
$\fs_\mu^{(k)}$ is a Schur function (Theorem~\ref{t schur kschur})
and deduce a tableau enumeration formula for a new class of Gromov-Witten invariants (Theorem \ref{cor:gw2schurtimeskschur}).
Positive formulas for Gromov-Witten invariants 
have been obtained in many special cases
\cite{BergBergeronThomasZabrocki, BergSaliolaSerranonearrectangle,BergSaliolaSerranoexpansions,  BKPT, CiocanFontanine,Coskun,FGP, MPP,MorseSchilling,Postnikovquantumversion, Tudose}.
Our formula is one of the few for which all three input permutations are allowed to vary
among $\Omega(2^k)$ many possibilities, in contrast to, say, the quantum Monk or Pieri formula in which one of the permutations has a very special form.


\textbf{Schur positivity and parabolic Hall-Littlewood polynomials.}
Conjecture~\ref{cj kschur pos} strengthens the following conjecture of Chen-Haiman \cite{ChenThesis} (by Remark \ref{r kschur conjecture}):

\begin{conjecture}
\label{con:Catalan2Schur}
The Catalan function $H(\Psi;\mu)$ is Schur positive for
any root ideal $\Psi$ and partition $\mu$.
\end{conjecture}

Broer and Shimozono-Weyman \cite{BroerNormality,SW} had earlier studied Conjecture~\ref{con:Catalan2Schur}
for {\it parabolic Hall-Littlewood polynomials}, those Catalan functions of the form
\begin{equation}
\label{parabolic}
H(\varnothing_{r_1}\uplus\cdots\uplus\varnothing_{r_d};\mu)\,
\end{equation}
with $r_1, \dots, r_d \in \ZZ_{\ge 1}$ and partition  $\mu$, where  $\varnothing_r \subset \Delta^+_r$ denotes
the empty root ideal of length  $r$.
Broer posed Conjecture \ref{con:Catalan2Schur} for this class of Catalan functions, or rather the stronger conjecture that
the higher cohomology of an associated vector bundle on the flag variety vanishes (see \cite[Conjecture~5]{SW}).
Shimozono-Weyman conjectured that the Schur expansion coefficients
of parabolic Hall-Littlewood polynomials can be described by a combinatorial procedure called katabolism~\cite{SW};
this conjecture was later generalized by Chen-Haiman to address any Catalan function with partition weight
\cite[Conjecture~5.4.3]{ChenThesis}.

Conjecture~\ref{cj kschur pos} predicts that the parabolic Hall-Littlewood polynomials
are not only Schur positive
but in fact $k$-Schur positive for  $k = \max\{(\mu^i)_1 + r_i -1 \mid i \in [d]\}$,
where  $\mu = (\mu^1, \dots, \mu^d)$ is the decomposition of $\mu$
into sequences of lengths $r_1, \dots, r_d$.
We settle this conjecture
for a subclass of parabolic Hall-Littlewood polynomials studied in \cite{LMksplit, LMirred}
called  $k$-split polynomials (Theorem \ref{t k split to k schur}).


\subsection{Unifying the $k$-Schur candidates}
\label{ss unifying}

The desired properties (i)--(iii) of a $k$-Schur basis were never simultaneously
satisfied by any one proposed candidate.  For this reason, the unification of the various definitions
has been an important open problem.
Table~\ref{table} summarizes
the current state of the art in this regard.

Circled entries follow from our
recent work in \cite{BMPS}, Theorem \ref{c three defs agree},  and Theorem \ref{t HL to k schur intro}.
The  $k$-Schur candidate $\{\tilde{A}^{(k)}_\lambda\}$
was defined via  $k$-split polynomials  in \cite{LMksplit, LMirred} (recalled in \S\ref{ss kschur via ksplit});
they were shown to form a basis for  $\Lambda^k$ in \cite{LMksplit}
and satisfy the $k$-rectangle property in \cite{LMirred}.
The candidate $\{s^{(k)}_\lambda\}$ was proposed in
\cite[\S9.3]{LLMSMemoirs1} (see also~\cite{AssafBilley,LLMSMemoirs2}) and is defined
as a sum of monomials over strong tableaux, certain chains in the strong (Bruhat) order of $\eS_{k+1}$.
The candidates on the last two rows are
  only defined at $t=1$
(and the checks indicate properties at $t=1$).
For these candidates, Schur positivity  was proven geometrically in \cite{LamBulletin}
and combinatorially in~\cite{LLMSMemoirs1,LLMSMemoirs2};
the  $k$-rectangle property was proven in \cite{LMktableau}.
The remaining checkmarks follow directly from the definitions of the respective candidates and the equivalence of the last two candidates \cite{LamSchubert}.

\begin{table}
\centerfloat
\vspace{1mm}
{\small
\begin{tabular}{lc|cccc}
\quad \qquad \qquad $k$-Schur candidate
& & \parbox{1.2cm}{basis \\ of  $\Lambda^k$} \!\!\!\!\!\!\! & \parbox{1.4cm}{Schur \\ positive} \!\!\!\!\! &
\parbox{3cm}{the $H_\mu(\mathbf x;q,t)$  are \\  $k$-Schur positive} \!\!\! & \parbox{2cm}{$k$-rectangle \\ property }  \!\!\!\!\!\! \\[2.7mm]
\hline &&&&& \\[-3.5mm]
\!\!\! (1998) Young tableaux and katabolism  \cite{LLM}  \!\!\!\!\!\!\!\!\!\!\!\!
& & & \checkmark &  &\\[1.5mm]
\hline &&&&& \\[-3.5mm]
\!\!\! (2001) {\small $k$-split polynomials \cite{LMksplit,LMirred}}
& $\tilde{A}^{(k)}_\lambda$
 & \checkmark  & \encircle{\checkmark} & \encircle{\checkmark} \ ($q=0$) & \checkmark \\[1.5mm]
\hline &&&&& \\[-3.5mm]
\!\!\! (2006) {\small Strong tableaux \cite{LLMSMemoirs1}}
& $s^{(k)}_\lambda$
& \encircle{\checkmark} &  \encircle{\checkmark}  &  \encircle{\checkmark}  \ ($q=0$) & \encircle{\checkmark} \\[1.5mm]
\hline &&&&& \\[-3.5mm]
\!\!\! (2008) {\small Catalan functions}
{\small  \cite{ChenThesis}}
& $\fs^{(k)}_\lambda$
& \encircle{\checkmark} & \encircle{\checkmark} & \encircle{\checkmark} \ ($q=0$)
& \encircle{\checkmark}
\\[1.5mm]
\hline &&&&& \\[-3.5mm]
\!\!\! (2014) Affine Kostka matrix
\cite{DalalMorse,LapointePinto}  \!\!\!\!\!\!
&
& \checkmark & & \checkmark \ ($q=0$) &  \\[1.5mm]
 \hline &&&&& \\[-3.5mm]
\!\!\! (2004) Weak tableaux  \cite{LMktableau}  ($t=1$)\!\!\!
& $\sone^{(k)}_\lambda$
& \checkmark & \checkmark & \checkmark \ ($q=0$) & \checkmark \\[1.5mm]
\hline &&&&& \\[-2.8mm]
\!\!\! (2005) Schubert classes in $H_*(\Gr)$  \cite{LamSchubert}
  ($t=1$)\!\!\!\!\!\!\!\!\!\!\!\!\!\!\!\!\!\!\!\!\!\!\!\!\!\!\!\!\!\!
&
& \checkmark & \checkmark & \checkmark \ ($q=0$) & \checkmark \\
\end{tabular}
}
\vspace{4mm}
\caption{\label{table}
\small Conjecturally equivalent definitions of  $k$-Schur functions and known properties.}
\vspace{-1mm}
\end{table}

\begin{theorem}
\label{c three defs agree}
The $k$-Schur functions defined via $k$-split polynomials \cite{LMksplit, LMirred},
$k$-Schur Catalan functions, 
and strong tableau $k$-Schur functions \cite[\S9.3]{LLMSMemoirs1}  coincide:
\begin{align}
\label{ec three defs agree}
\tilde{A}^{(k)}_\mu({\bf x};t) = \fs^{(k)}_\mu({\bf x};t)  =s^{(k)}_\mu({\bf x};t) \quad \text{ for all $\mu \in \Par^k$}.
\end{align}
Moreover, their  $t=1$ specializations $\{s^{(k)}_\mu({\bf x};1)\}$
match a definition  $\{\sone^{(k)}_\mu(\mathbf{x})\}$
using weak tableaux
from \cite{LMktableau},  and represent Schubert classes in
the homology of the
affine Grassmannian $\Gr_{G}$ of  $G=SL_{k+1}$.
\end{theorem}
\begin{proof}
The first equality of \eqref{ec three defs agree} is proved in Section \ref{s ksplit proofs} as a consequence of the abovementioned Theorem \ref{t k split to k schur},
while the second was established in \cite[Theorem 2.4]{BMPS}.
The $\{s^{(k)}_\mu({\bf x};1)\}$ agree with the weak tableau  $k$-Schur functions by
\cite[Theorem 4.11]{LLMSMemoirs1} and with affine Grassmannian Schubert classes by \cite[Theorem 7.1]{LamSchubert}.
\end{proof}

\section{Strong Pieri operators}

The strong Pieri operators were introduced in~\cite{BMPS} and played a peripheral role there.
We have since discovered they are key to establishing elegant
formulas for  $k$-Schur expansions of Catalan functions.
The operators are defined combinatorially using strong marked
covers and have a  simple description in terms of Catalan functions.

For most of the background below, we follow \cite{BMPS}.  Strong marked covers were introduced in
\cite{LLMSMemoirs1} and our version below differs only in that markings are by rows rather than
diagonals; this difference, though seemingly minor,  has turned out to be quite important.
For examples of the subsequent definitions, see \cite[\S2.2]{BMPS}
and Examples \ref{ex k Schur to Schur} and \ref{ex schur times kschur}.

%

\subsection{Strong tableaux}
\label{ss Strong marked tableaux}

The \emph{diagram} of a partition $\lambda$ is the subset of boxes
$\{(r,c) \in \ZZ_{\geq 1} \times \ZZ_{\geq 1} \mid c \le \lambda_r\}$
in the plane, drawn in English (matrix-style) notation so that rows (resp. columns)
are increasing from north to south (resp. west to east).
Each box has a \emph{hook length} which counts the
number of boxes below it in its column and weakly to its right in its row.
A {\it $k+1$-core} is a partition with no box of hook length $k+1$.
There is a bijection $\p$ \cite{LMtaboncores}
from the set of $k+1$-cores to $\Par^k$
mapping a $k+1$-core $\kappa$ to the partition  $\lambda$ whose $r$-th row  $\lambda_r$ is
the number of boxes in the  $r$-th row of $\kappa$ having hook length $\le k$.

A \emph{strong cover} $\tau \Rightarrow \kappa$ is a pair of $k+1$-cores such that $\tau \subset \kappa$ and $|\p(\tau)| + 1 = |\p(\kappa)|$.
A \emph{strong marked cover} $\tau \xRightarrow{~~r~~} \kappa$ is a strong cover  $\tau \Rightarrow \kappa$
together with a positive integer~$r$
which is allowed to be the smallest row index
 of any connected component of the skew shape $\kappa/\tau$.
Let $w = w_1 \cdots w_m$ be a word in the alphabet of positive integers.
 A \emph{strong tableau  $T$ marked by $w$}
is a sequence of strong marked covers of the form
\[\kappa^{(0)} \xRightarrow{~~w_m~~} \kappa^{(1)} \xRightarrow{~~w_{m-1}~~} \cdots  \xRightarrow{~~w_1~~} \kappa^{(m)}.\]
We write $\inside(T) = \p(\kappa^{(0)})$ and $\outside(T) = \p(\kappa^{(m)})$.
The set of strong tableaux marked by  $w$ with $\outside(T) = \mu$ is denoted $\SMT^k(w \, ; \mu)$.

The \emph{spin} of a strong marked cover $\tau \xRightarrow{~~r~~} \kappa$ is defined to be $c\cdot (h-1) + N$, where $c$ is the number of connected components of  the skew shape $\kappa/\tau$, $h$ is the height (number of rows) of each component, and $N$ is the number of components entirely contained in rows $>r$.
For  a strong tableau  $T$ marked by a word, $\spin(T)$ is defined to be the sum of the spins of the strong marked covers comprising $T$.

\subsection{Strong Pieri operators}
\label{ss strong Pieri ops}

Fix a positive integer $k$.
The \emph{strong Pieri operators} $u_1, u_2, \dots \in \End_{\ZZ[t]}(\Lambda^k)$
are defined by their action on the basis  $\{\fs^{(k)}_\mu\}_{\mu \in \Par^k}$ as follows:
\begin{align}
\label{e strong Pieri operator}
\fs^{(k)}_\mu \cdot u_p =
\sum_{T \in \SMT^k(p \, ; \mu)} t^{\spin(T)}\fs^{(k)}_{\inside(T)}\, .
\end{align}
We have found it more natural to define these as right operators for compatibility with conventions for tableau reading words (see Theorem~\ref{t schur kschur}).
The set  $\SMT^k(p \, ; \mu)$ in the sum is just another notation for
the set of strong marked covers  $\tau \xRightarrow{~~p~~} \kappa$ with  $\p(\kappa) = \mu$.

By \cite[Theorem 9.2]{BMPS}, these operators are simply described in terms of Catalan functions:
for any $\mu \in \Par^k_\ell$ and  $p \in [\ell]$,
\begin{align}
\label{ec flm k schur}
\fs^{(k)}_\mu \cdot u_p &=  H(\Delta^k(\mu) ; \mu-\epsilon_p).
\end{align}

Let  $e_d^\perp$ be the linear operator on $\Lambda$ that is adjoint to multiplication by $e_d$ with respect to the Hall inner product.
A main result of \cite{BMPS} (see Equation 2.6 and \S9.2 therein) expresses  $e_d^\perp$ in terms of the strong Pieri operators.  We need only the following special case:
\begin{theorem}
\label{t u perp}
For any   $f \in \Span_{\ZZ[t]}\!\big\{\fs^{(k)}_{\mu}(\mathbf{x};t) \mid \mu \in \Par_\ell^k \big\}$,
$f \cdot e_{\ell}^\perp = f \cdot u_{\ell} u_{\ell-1} \cdots u_1$.
\end{theorem}

\begin{remark}
\label{r integrality}
In \cite{BMPS}, we worked with symmetric functions over the coefficient ring  $\QQ(t)$
rather than  $\ZZ[t]$, but it is easily checked that Theorem \ref{t u perp} and all other results of \cite{BMPS} cited here
hold over  $\ZZ[t]$.
\end{remark}

\section{Catalan operators}

Our results on modified Hall-Littlewood polynomials and $k$-split polynomials make use of symmetric
function operators of Jing and Garsia \cite{Jing, Garsiaop} and Shimozono-Zabrocki \cite{SZ}.
It is natural from our perspective to frame these in the context of more general Catalan operators,
which recover Catalan functions upon action on 1.

Garsia's version \cite{Garsiaop} of Jing's Hall-Littlewood vertex operators \cite{Jing}
are the symmetric function operators defined for any $m\in \ZZ$ by
\begin{align}
\label{ed Jing}
 \bb_m = \sum_{i,j \ge 0}  (-1)^i t^j h_{m+i+j}(\mathbf{x}) e_i^\perp h_j^\perp \, \in \End_{\ZZ[t]}(\Lambda).
\end{align}
These are creation operators for the modified Hall-Littlewood polynomials:
\begin{equation}
\label{eq:HLJing}
\bb_m \, H_{\mu}(\mathbf x;t)=
H_{(m,\mu)}(\mathbf x;t)  ~~\ \ \text{ for $m \ge \mu_1$}.
\end{equation}

For $\alpha \in \ZZ^\ell$, set $\tilde{\bb}_\alpha := \bb_{\alpha_1}\bb_{\alpha_2}\cdots\bb_{\alpha_\ell}$.

\begin{definition}
\label{d Catalan operator}
The \emph{Catalan operator} associated to an indexed root ideal $(\Psi, \gamma)$ is
the symmetric function operator given by
\begin{align}
\label{ed HH definition CHL}
\HHH^\Psi_\gamma \, &= \prod_{(i,j) \in \Delta^+ \setminus \Psi} (1-t\mathbf{R}_{i j})\tilde{\bb}_\gamma  \ \in \End_{\ZZ[t]}(\Lambda),
\end{align}
where the raising operator $\mathbf{R}_{i j}$ acts on the subscripts of the $\tilde{\bb}_\alpha$ by $\mathbf{R}_{i j}\tilde{\bb}_\alpha = \tilde{\bb}_{\alpha + \epsilon_i - \epsilon_j}$.
\end{definition}

Letting Catalan operators act on 1, we recover Catalan functions:
\begin{align}
\label{e Catalan op eq fun}
\HHH^\Psi_\gamma \cdot 1 = \HRI   \quad \text{ for any indexed root ideal  $(\Psi, \gamma)$}.
\end{align}
This holds by \cite[Proposition 4.7]{BMPS}.
The Catalan operators simultaneously generalize  the iterated Garsia-Jing operators $\tilde{\bb}_\alpha$ and
the generalized Hall-Littlewood vertex operators  $\bb_\alpha$ of Shimozono-Zabrocki \cite{SZ}:
for any  $\alpha \in \ZZ^\ell$,
\begin{align}
\text{$\tilde{\bb}_\alpha  =
 \HHH^{\Delta^+}_\alpha$ \quad and \quad \ $\bb_\alpha = \HHH^\varnothing_\alpha$ \,.}
\end{align}
The latter follows from, e.g., the description \cite[Equation 6.7]{LMksplit} of the  $\bb_\alpha$ operators.

\begin{proposition}
\label{p Catalan box times}
Catalan operators obey the following composition law:
for any indexed root ideals $(\Psi, \mu)$ and  $(\Phi, \nu)$ of lengths  $r$ and $\ell-r$, respectively, there holds
\begin{align}
\label{ep Catalan box times}
\HHH^{\Psi \,}_\mu \HHH^{\Phi\,}_\nu = \HHH^{\Psi \uplus\Phi \, }_{\mu\nu}\,,
\end{align}
where $\uplus$ is defined in \eqref{ed root ideal uplus}.
\end{proposition}

\begin{proof}
We compute using Definition \ref{d Catalan operator}\,:
\begin{align*}
 \HHH^{\Psi\uplus \Phi\,}_{ \mu \nu} &=  \prod_{(i,j) \in \Delta_\ell^+ \setminus (\Psi\uplus \Phi)} (1-t\mathbf{R}_{i j})\tilde{\bb}_{\mu\nu}   \\
&= \bigg( \prod_{(i,j) \in \Delta_r^+ \setminus \Psi} (1-t\mathbf{R}_{i j})\tilde{\bb}_{\mu}\bigg)  \bigg(  \prod_{(i,j) \in \Delta_{\ell-r}^+ \setminus \Phi}
 (1-t\mathbf{R}_{i j})\tilde{\bb}_{\nu}\bigg) = \HHH^\Psi_\mu \, \HHH^{\Phi \, }_\nu. \  \qedhere
\end{align*}
\end{proof}

The  $t=1$ specializations of Catalan operators and functions can be made precise as follows:
the ring homomorphism  $\ZZ[t] \to \ZZ$, $t \mapsto 1$
makes  $\ZZ$ a  $\ZZ[t]$-algebra
and  $\ZZ \tsr_{\ZZ[t]} \!-$ is a
functor from  $\ZZ[t]$-$\Mod$ to $\ZZ$-$\Mod$ which we denote  by  $|_{t=1}$.
Let $\Lambda_{\ZZ} = \Lambda|_{t=1} = \Lambda/(t-1) = \ZZ[h_1, h_2, \dots]$.
Specializing $t=1$ at the level of elements of $\Lambda$ is then defined via the canonical
ring homomorphism
$\pi: \Lambda \to \Lambda/(t-1) = \Lambda_\ZZ$\,:
\[ H(\Psi ; \gamma)(\mathbf{x};1 ) := \pi(H(\Psi ; \gamma)(\mathbf{x};t ))
\  \ \text{  and  } \ \
\fs^{(k)}_\mu(\mathbf{x};1) := \pi(\fs^{(k)}_\mu(\mathbf{x};t)).
\]

\begin{proposition}
\label{p Catalan mult}
At  $t=1$, Catalan operators reduce to multiplication by Catalan functions:
for any  $g \in \Lambda_\ZZ$,
$\HHH^{\Psi \,}_\mu|_{t=1} (g) = (\HHH^{\Psi \,}_\mu|_{t=1} \cdot 1) \, g = H(\Psi ;\mu)(\mathbf{x};1) \, g$.
\end{proposition}
\begin{proof}
The first equality follows from the fact that $\bb_m|_{t=1} \in \End_\ZZ(\Lambda_\ZZ)$ is equal to multiplication
by $h_m(\mathbf{x})$ (recall $h_m(\mathbf{x}):=0$ for $m<0$).
The second holds by \eqref{e Catalan op eq fun} and the general fact that for any  $\bb \in \End_{\ZZ[t]}(\Lambda)$ and  $f \in \Lambda$,
$\pi(\bb (f)) = \bb|_{t=1}(\pi(f))$.
\end{proof}

\begin{corollary}
\label{c Catalan mult t1}
For any indexed root ideals $(\Psi, \mu)$ and  $(\Phi, \nu)$,
\begin{align}
\label{ec Catalan mult t1}
H(\Psi; \mu)(\mathbf{x};1)  H(\Phi; \nu)(\mathbf{x};1) = H(\Psi \uplus\Phi; \mu\nu)(\mathbf{x};1).
\end{align}
\end{corollary}
\begin{proof}
Apply the functor $|_{t=1}$ to $\HHH^{\Psi \,}_\mu \HHH^{\Phi\,}_\nu = \HHH^{\Psi \uplus\Phi \, }_{\mu\nu}$,
use Proposition \ref{p Catalan mult}, and act on~1.
\end{proof}


\section{Modified Hall-Littlewood polynomials}
\label{s HL into kschur}

We give a positive combinatorial formula for the  $k$-Schur expansion
of the modified Hall-Littlewood polynomials $H(\Delta^+;\mu)=H_\mu(\mathbf {x};t)$,
which is succinctly expressed in terms of the strong Pieri operators.
This resolves the $q=0$ specialization of the strengthened Macdonald positivity conjecture \cite{LLM}.

For a skew shape  $\theta$, the \emph{superstandard tableau $Z_{\theta}$} is the
unique filling of  $\theta$ whose  $i$-th row consists entirely of $i$'s.
The \emph{column reading word} of a tableau $T$, denoted $\creading(T)$,
is the word obtained by concatenating the columns of $T$ (reading each column bottom to top), starting with the leftmost column.
%
For example, with $\theta=(44444)/(43220)$,
\[
Z_{\theta}=
{\fontsize{6pt}{5pt}\selectfont
\tableau{
\bl &\bl &\bl &\fr[r]  \\
\bl &\bl&\bl& 2 \\
\bl &\bl & 3&3\\
\bl &\bl & 4 & 4\\
5 & 5 & 5 & 5
}} \qquad \text{ and } \quad
\creading(Z_{\theta}) = 555435432
\, .\]

Given a word  $w = w_1 \cdots w_d$ in the positive integers, we write $u_w = u_{w_1} \cdots u_{w_d}$ for
the corresponding monomial in the strong Pieri operators.

\begin{theorem}
\label{t HL to k schur intro}
For any $\mu \in \Par^k_\ell$,
the $k$-Schur expansion of the modified Hall-Littlewood polynomial $H_\mu(\mathbf{x};t)$ is given by
\begin{align*}
H_\mu &= \fs^{(k)}_{k^\ell} \cdot u_{\creading(Z_{k^{\ell}/\mu})} = \sum_{T \in \SMT^k\big(\creading(Z_{k^{\ell}/\mu})\, ; \, k^\ell\big)} t^{\spin(T)} \fs^{(k)}_{\inside(T)}.
\end{align*}
\end{theorem}

We give the proof now though it requires a result proved later, Lemma~\ref{l commute with square},
which describes the interaction of the strong Pieri operators with
the Garsia-Jing operators.

\begin{proof}
Set $\theta=k^\ell/\mu$.
The proof is by induction on $\ell + |\theta|$.
The base case
$\ell = 0$ holds by $H_{\emptyset} = 1 = \fs^{(k)}_\emptyset \cdot u_{\creading(Z_\theta)}$ ($\creading(Z_\theta)$ is the empty word).
Now suppose  $\ell > 0$.
If $\mu_1 < k$, then the rightmost column of $Z_\theta$ is a full column of length $\ell$, and so
$u_{\creading(Z_\theta)} = u_{\creading(Z_{\tilde{\theta}})} u_{\ell} u_{\ell-1} \cdots u_1$, where  $\tilde{\theta} := k^\ell/(\mu+1^\ell)$.
By the inductive hypothesis,
\begin{align}
H_{\mu + 1^\ell} = \fs^{(k)}_{k^\ell} \cdot u_{\creading(Z_{\tilde{\theta}})}. \label{e lambda plus one ell}
\end{align}
Applying $e_\ell^\perp$ to both sides and using Theorem \ref{t u perp} gives
\[H_{\mu + 1^\ell} \cdot e_\ell^\perp =
\fs^{(k)}_{k^\ell} \cdot u_{\creading(Z_{\tilde{\theta}})} u_{\ell} u_{\ell-1} \cdots u_{1} = \fs^{(k)}_{k^\ell} \cdot u_{\creading(Z_\theta)}.\]
Now $H(\Psi;\gamma) \cdot e_\ell^\perp  = H(\Psi; \gamma-1^\ell)$ for any indexed root ideal  $(\Psi, \gamma)$ of length  $\ell$, which  follows directly from
Definition \ref{d HH gamma Psi} and the fact that $s_\gamma \cdot e_\ell^\perp = s_{\gamma-1^\ell}$ for any  $\gamma \in \ZZ^\ell$.
Hence  $H_{\mu + 1^\ell} \cdot e_\ell^\perp = H_\mu$, completing the proof in the  $\mu_1 < k$ case.

Now suppose $\mu_1 = k$ and set $\hat{\theta} = k^{\ell-1}/(\mu_2, \dots, \mu_\ell)$.
Then $\creading(Z_\theta)$ does not contain any 1's and is obtained from $\creading(Z_{\hat{\theta}})$ by adding 1 to each letter.
Thus Lemma~\ref{l commute with square} with $r=1$ yields
\begin{align*}
\fs^{(k)}_{k^\ell} \cdot u_{\creading(Z_\theta)} = \bb_k \mysquare \big( \fs^{(k)}_{k^{\ell-1}} \cdot \myhat{u}_{\creading(Z_{\hat{\theta}})} \big)
= \bb_k \mysquareb H_{(\mu_2, \dots, \mu_\ell)} = H_\mu,
\end{align*}
where the second equality is by the inductive hypothesis and the third is by \eqref{eq:HLJing}.
\end{proof}

\begin{example}
\label{ex k Schur to Schur}
According to Theorem \ref{t HL to k schur intro},
the 3-Schur expansion of $H_{2211}$
is given by $\fs^{(3)}_{3333} \cdot u_4 u_3 u_4 u_3 u_2 u_1
=  \displaystyle \sum_{\SMT^3(434321 \, ; \, 3333)} \!\!\!\! t^{\spin(T)}\fs^{(3)}_{\inside(T)}$;  this sum is over sequences of the form
\[\kappa^{(0)} \xRightarrow{~~1~~} \kappa^{(1)} \xRightarrow{~~2~~} \kappa^{(2)} \xRightarrow{~~3~~}  \kappa^{(3)}  \xRightarrow{~~4~~} \kappa^{(4)}
\xRightarrow{~~3~~} \kappa^{(5)}
\xRightarrow{~~4~~} \kappa^{(6)} = \p^{-1}(3333)\, ,\]
which are illustrated below by placing letter $i$ in $\kappa^{(i)}/\kappa^{(i-1)}$.
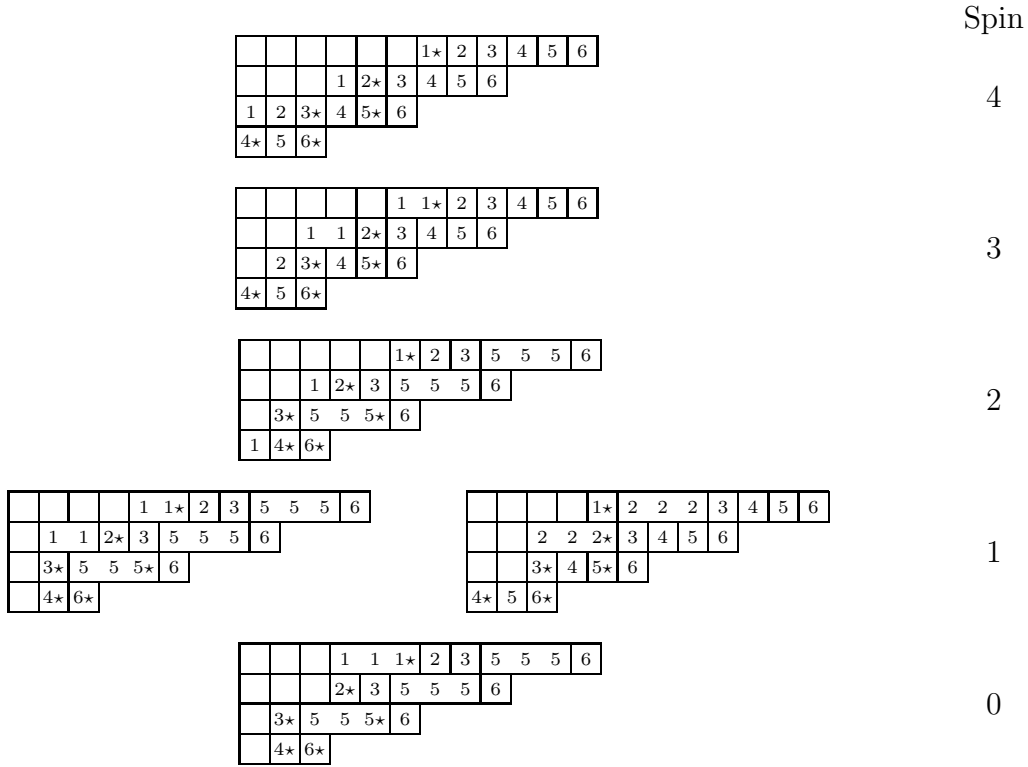
\[\begin{tikzpicture}[xscale = .42,yscale = 2.02]
\tikzstyle{vertex}=[inner sep=3pt, outer sep=4pt]
\tikzstyle{framedvertex}=[inner sep=3pt, outer sep=4pt]
\tikzstyle{aedge} = [draw, thin, ->,black]
\tikzstyle{edge} = [draw, thick, -,black]
\tikzstyle{doubleedge} = [draw, thick, double distance=1pt, -,black]
\tikzstyle{hiddenedge} = [draw=none, thick, double distance=1pt, -,black]
\tikzstyle{dashededge} = [draw, very thick, dashed, black]
\tikzstyle{LabelStyleH} = [text=black, anchor=south]
\tikzstyle{LabelStyleHn} = [text=black, anchor=north]
\tikzstyle{LabelStyleV} = [text=black, anchor=east]
\setlength{\cellsize}{1.6ex}
\begin{scope}[xscale = 1.8][xshift = 540pt]
\node[vertex] at (0,5) {\fontsize{7.5pt}{5.7pt}\selectfont $\text{ \tableau{
~&~&~&~&~&~&\crc{1}&2&3&4&5&6\\~&~&~&1&\crc{2}&3&4&5&6\\1&2&\crc{3}&4&\crc{5}&6\\\crc{4}&5&\crc{6}\\}
}$};
\node[vertex] at (0,4) {\fontsize{7.5pt}{5.7pt}\selectfont $\text{
\tableau{&&&&&\fr[l,t]1&\fr[t,r]\crc{1}&2&3&4&5&6\cr
&&\fr[l,t]1&\fr[t,r]1&\crc{2}&3&4&5&6\cr
&2&\crc{3}&4&\crc{5}&6\cr
\crc{4}&5&\crc{6}}
}$};
\node[vertex] at (0,3) {\fontsize{7.5pt}{5.7pt}\selectfont $\text{
\tableau{
&&&&&\crc{1}&2&3&\fr[l,t]5&\fr[t,b]5&\fr[t,b,r] 5&6\cr
&&1&\crc{2}&3&\fr[l,t]5&\fr[t,b]5&\fr[t,b,r]5&6\cr
&\crc{3}&\fr[l,t]5&\fr[t,b]5&\fr[t,b,r]\crc{5}&6\cr
1&\crc{4}&\crc{6}
}}$};
\node[vertex] at (4,2) {\fontsize{7.5pt}{5.7pt}\selectfont $\text{
\tableau{&&&&\crc{1}&\fr[t,l]2&\fr[t]2&\fr[t,r]2&3&4&5&6\cr
&&\fr[t,l]2&\fr[t]2&\fr[t,r]\crc{2}&3&4&5&6\cr
&&\crc{3}&4&\crc{5}&6\cr
\crc{4}&5&\crc{6}
}}$};
\node[vertex] at (-4,2) {\fontsize{7.5pt}{5.7pt}\selectfont $\text{
\tableau{
&&&&\fr[t,l]1&\fr[t,r]\crc{1}&2&3&\fr[l,t]5&\fr[t,b]5&\fr[t,b,r] 5&6\cr
&\fr[t,l]1&\fr[t,r]1&\crc{2}&3&\fr[l,t]5&\fr[t,b]5&\fr[t,b,r]5&6\cr
&\crc{3}&\fr[l,t]5&\fr[t,b]5&\fr[t,b,r]\crc{5}&6\cr
&\crc{4}&\crc{6}
}
}$};
\node[vertex] at (0,1) {\fontsize{7.5pt}{5.7pt}\selectfont $\text{
\tableau{
&&&\fr[t,l]1&\fr[t]1&\fr[t,r]\crc{1}&2&3&\fr[t,b,l]5&\fr[t,b]5&\fr[t,b,r]5 &6\cr
&&&\crc{2}&3&\fr[t,l]5&\fr[t,b]5&\fr[t,b,r]5 &6\cr
&\crc{3}&\fr[t,l]5&\fr[t,b]5&\fr[t,b,r]\crc{5} &6\cr
&\crc{4}&\crc{6}
}}$};
\end{scope}
\begin{scope}[xshift = 520pt]
\node at (0,5.5) {$\text{Spin}$};
\node at (0,5) {4};
\node at (0,4) {3};
\node at (0,3) {2};
\node at (0,2) {1};
\node at (0,1) {0};
\end{scope}
\end{tikzpicture}\]
\[H_{2211} = t^4\, \fs^{(3)}_{33} + t^3\, \fs^{(3)}_{321} + t^2\, \fs^{(3)}_{321}+t\, \fs^{(3)}_{3111} + t\, \fs^{(3)}_{222} + \fs^{(3)}_{2211}.
\]
\end{example}

See also \cite[Figure 1]{BMPS}  for the 4-Schur expansion of $H_{1111}$.

\section{Schur times  $k$-Schur}

We give a positive combinatorial formula for the $k$-Schur expansion of a  $t$-analog of a Schur function times a $k$-Schur function when the
indexing partitions concatenate to a partition.
The involved combinatorics has the same spirit as
the $k$-Schur expansion of modified Hall-Littlewood polynomials,
but rather than using the action of strong Pieri operators derived
from one particular tableau word, it requires a set of tableaux words.


Let  $\SSYT_\theta(\A)$ denote the set of semistandard Young tableaux of skew shape  $\theta$ with entries from an alphabet  $\A$
(fillings of the diagram of  $\theta$ which are weakly increasing in rows and strictly increasing down columns).


\begin{theorem}
\label{t schur kschur}
For $\mu \in \Par^{k-r+1}_r$ and $\nu \in \Par^k$ such that $\mu\nu$ is a partition,
\begin{align}
\label{et schur kschur}
\bb_\mu \mysquareb \fs^{(k)}_\nu = \sum_{T \in \SSYT_{U/\mu}([r])} \fs^{(k)}_{U\nu} \cdot u_{\creading(T)} \, ,
\end{align}
where $U := (k-r+1)^r$ and  $\creading(T)$ is  as defined before Theorem~\ref{t HL to k schur intro}.
\end{theorem}

This result combinatorially describes the class of Gromov-Witten invariants claimed in \S\ref{ss kSchur positivity conjecture} (see also Theorem~\ref{cor:gw2schurtimeskschur});
this is because when $t=1$, $\bb_\mu \mysquareb \fs^{(k)}_\nu$
reduces to the product $s_\mu \fs^{(k)}_\nu(\mathbf{x};1)$, and the condition $\mu \in \Par^{k-r+1}_r$
is equivalent to $\mu$ belonging to $\Par^k_r$ and
$\fs^{(k)}_\mu = s_\mu$.  The proof of Theorem~\ref{t schur kschur}
is given in Section \ref{s schur kschur proof}.

\begin{remark}
\label{r no pos}
While the product $s_\mu \fs^{(k)}_\nu(\mathbf{x};1)$
is  $k$-Schur positive
for any  $\mu \in \Par^{k-r+1}_r$ and $\nu \in \Par^k$ (see Section \ref{s gromov witten}),
positivity often fails for its $t$-analog $\bb_\mu \, \fs^{(k)}_\nu$\,:
$\bb_{1} \mysquareb \fs^{(4)}_3 = t^3\fs^{(4)}_4 +t^2\fs^{(4)}_{31}+(t-1)\fs^{(4)}_{22}$.
We believe that for $\bb_\mu \, \fs^{(k)}_\nu$ to be  $k$-Schur positive,
the condition that  $\mu\nu$ is a partition is
close to best possible.
This example also shows that positivity can fail even if all  $k$-Schur functions are Schur functions.
\end{remark}

\begin{example}
\label{ex schur times kschur}
Let  $k = 6$, $r=3$, $\mu = 432$,  $\nu = 211111$.  Then  $U = 444$  and
\[\SSYT_{U/\mu}([r]) =\fontsize{6pt}{5pt}\selectfont
\Bigg\{\
\tableau{
\bl& \fr[r] \\
\bl& 1 \\
 1 & 2
}
\quad \tableau{
\bl& \fr[r] \\
\bl& 1 \\
 1 & 3
}
\quad \tableau{
\bl& \fr[r] \\
\bl& 2 \\
 1 & 3
}
\quad \tableau{
\bl& \fr[r] \\
\bl& 1 \\
2& 2
}
\quad\tableau{
\bl& \fr[r] \\
\bl& 1 \\
 2 & 3
}
\quad \tableau{
\bl& \fr[r] \\
\bl& 2 \\
 2 & 3
}
\quad \tableau{
\bl& \fr[r] \\
\bl& 1 \\
 3 & 3
}
\quad \tableau{
\bl& \fr[r] \\
\bl& 2 \\
3 & 3
}\
\Bigg\}.
 \]
Theorem \ref{t schur kschur} gives
\[\bb_\mu \mysquareb \fs^{(k)}_\nu = \fs^{(k)}_{U\nu} \cdot \big( u_{121} + u_{131} + u_{132} + u_{221} + u_{231} + u_{232} + u_{331} + u_{332} \big). \]
This yields the strong tableaux (note $U\nu = 444211111  = \p(755311111)$)
\begin{align*}
\begin{array}{ccccccc}
\!\!\!\!\!\!T
& \!{\fontsize{5.8pt}{4.8pt}\selectfont \tableau{~&~&~&~&~&\crc{1}&\crc{3}\\~&~&~&~&~\\~&~&~&~&\crc{2}\\~&1&3\\~\\~\\~\\~\\3\\}}
& \! {\fontsize{5.8pt}{4.8pt}\selectfont \tableau{~&~&~&~&~&~&\crc{1}\\~&~&~&~&~\\~&~&~&\crc{2}&\crc{3}\\~&~&1\\~\\~\\~\\~\\1\\}}
& \! {\fontsize{5.8pt}{4.8pt}\selectfont \tableau{~&~&~&~&~&~&\crc{3}\\~&~&~&~&\crc{1}\\~&~&~&~&\crc{2}\\~&~&3\\~\\~\\~\\~\\3\\}}
& \! {\fontsize{5.8pt}{4.8pt}\selectfont \tableau{~&~&~&~&~&1\fr[l,t,b]&\crc{1}\fr[r,t,b]\\~&~&~&~&~\\~&~&~&\crc{2}&\crc{3}\\~&1\fr[l,t,b]&1\fr[r,t,b]\\~\\~\\~\\~\\~\\}}
& \! {\fontsize{5.8pt}{4.8pt}\selectfont \tableau{~&~&~&~&~&3\fr[l,t,b]&\crc{3}\fr[r,t,b]\\~&~&~&~&\crc{1}\\~&~&~&~&\crc{2}\\~&3\fr[l,t,b]&3\fr[r,t,b]\\~\\~\\~\\~\\~\\}}
& \! {\fontsize{5.8pt}{4.8pt}\selectfont \tableau{~&~&~&~&~&~&~\\~&~&~&~&\crc{1}\\~&~&~&\crc{2}&\crc{3}\\~&~&~\\~\\~\\~\\~\\~\\}}
\\[16mm]
\!\!\!\!\!\! \spin(T) & \!\! 3 & \!\! 2 & \!\! 2 &\!\! 1 &\!\! 1 &\!\! 0
\end{array}
\end{align*}
and summing $t^{\spin(T)} \fs^{(k)}_{\inside(T)}$ over these tableaux  gives
\begin{align}
\label{e Schur times kSchur}
 \bb_\mu \mysquareb \fs^{(k)}_\nu =
\bb_{432} \mysquareb \fs^{(6)}_{21^5} =
t^3 \fs^{(6)}_{4431^5} +
t^2 \fs^{(6)}_{44221^4} +
t^2 \fs^{(6)}_{43321^4} +
t \, \fs^{(6)}_{4421^6} +
t \, \fs^{(6)}_{4331^6} +
 \fs^{(6)}_{43221^5}.
\end{align}
\end{example}

A notable special case of Theorem \ref{t schur kschur}, namely, when $r = 1$ and then  $\mu = (d)$ with  $d\le k$, 
relates to the Pieri rule for $k$-Schur functions.

\begin{corollary}
\label{c Br on k schur}
Let $d \le k$ be positive integers and $\nu \in \Par^d$.
Then
\begin{align}
\label{ec Br on kschur}
\bb_d \mysquareb \fs^{(k)}_\nu = \fs^{(k)}_{(k, \nu)} \cdot u_{1}^{k-d}.
\end{align}
\end{corollary}

Even this special (Pieri) case of Theorem \ref{t schur kschur} is new and significant.  
Its specialization at $t=1$ was previously known, but even here the story is interesting:
the  $k$-Pieri rule from \cite{LMktableau} expresses $h_d \sone^{(k)}_{\nu}$ in the $k$-Schur basis $\{\sone^{(k)}_{\mu}\}$ using weak tableaux.
On the other hand, \eqref{ec Br on kschur} at  $t=1$ becomes $h_d \sone^{(k)}_{\nu} = \sone^{(k)}_{(k,\nu)} \cdot u_1^{k-d}|_{t=1}$,
which (when  $d \ge \nu_1$) agrees
with a formulation of the  $k$-Pieri rule from \cite[Corollary 16]{DalalMorse};
only after some work (done in \cite[Section 4]{DalalMorse}) can it be shown that these versions of the  $k$-Pieri rule indeed compute the same thing.

\section{$k$-split polynomials}
\label{s ksplit into kschur}

As mentioned in the introduction,
Conjecture \ref{cj kschur pos} predicts that
 for any tuple  $(\lambda^1, \dots, \lambda^d)$ of partitions which concatenate to a partition,
$\bb_{\lambda^1} \mysquare \bb_{\lambda^2}\mysquare  \cdots \mysquare \bb_{\lambda^{d-1}} \mysquareb s_{\lambda^d}$
is  $k$-Schur positive for $k = \max\{(\lambda^i)_1 + \ell(\lambda^i)-1 \mid i \in [d]\}$,
strengthening the corresponding Schur positivity conjecture of Broer and Shimozono-Weyman.
We resolve this (stronger) conjecture for the $k$-split polynomials,
a basis for  $\Lambda^k$ 
introduced in~\cite{LMksplit} to define the  
$k$-Schur candidate  $\{\tilde{A}^{(k)}_\mu \}$.
As a corollary, we deduce that the  $\tilde{A}^{(k)}_\mu$ agree with the  $k$-Schur Catalan functions.

\subsection{$k$-split polynomials are  $k$-Schur positive}
\label{ss ksplit into kschur}

For $\lambda \in \Par^k$,
the \emph{$k$-split of $\lambda$} is
 the sequence $\lambda^{\rightarrow k} := (\lambda^{1}, \lambda^{2}, \dots , \lambda^{d})$
obtained
by decomposing $\lambda$ (without rearranging entries)
into partitions
$\lambda^{i} \in \Par^{k}_{r_i}$ so that $r_i = k-(\lambda^i)_1+1$ and  $(\lambda^d)_1 > 0$.  Thus the maximum hook length of $\lambda^{i}$ is $k$ for all $i < d$
and $\lambda^{d}$ is padded with zeros so that
it has a total of $k-(\lambda^{d})_1+1$ entries. We adopt the convention that the  $k$-split of the empty partition is the empty list of partitions ($d= 0$).


For example, $(322211)^{\rightarrow 3} = (3,22,21,100)$
and $(322211)^{\rightarrow 4} =(32,22,1000)$:
\vspace{-2mm}

\begin{tikzpicture}[xscale = 1.6, yscale = .68]
\node at (-1.25,-.25) {\fontsize{3.8pt}{3.8pt}\selectfont $
\tableau{
 & & \\
 & \\
 & \\
 & \\
 \\
}$ };
\node at (-.65,-.25) {$\rightarrow^{(3)}$};
\node at (0,1.5) {\fontsize{3.8pt}{3.8pt}\selectfont$\tableau{\bl \\ & & }$};
\node at (0,.5) {\fontsize{3.8pt}{3.8pt}\selectfont $\tableau{ & &\bl \\  & }$};
\node at (0,-.5) {\fontsize{3.8pt}{3.8pt}\selectfont $\tableau{ & &\bl\\  & \bl}$};
\node at (0,-1.675) {\fontsize{3.8pt}{3.8pt}\selectfont $\tableau{ &\bl &\bl \\ \fr[l] \\ \fr[l]}$};
\node at (2,-.5) {\fontsize{3.8pt}{3.8pt}\selectfont
$\tableau{
 & &\\
 & \\
 & \\
 & \\
  \\
}$};
\node at (2.6,-.25) {$\rightarrow^{(4)}$};
\node at (3.25,1) {\fontsize{3.8pt}{3.8pt}\selectfont$\tableau{
\bl \\
 & &\\
 & \\}$};
\node at (3.25,0) {\fontsize{3.8pt}{3.8pt}\selectfont$\tableau{
 \bl \\
 & \\
 & & \bl }$};
\node at (3.25,-1.625) {\fontsize{3.8pt}{3.8pt}\selectfont$\tableau{
 &\bl & \bl \\
 \fr[l]\\
 \fr[l]\\
 \fr[l]
}$};
\node at (-3.4,0) {\ };
\end{tikzpicture}
\vspace{-2mm}

\begin{definition}
For $\lambda \in \Par^k$ with  $k$-split $\lambda^{\rightarrow k} = (\lambda^{1}, \lambda^{2}, \dots, \lambda^{d})$,
the \emph{$k$-split polynomial} indexed by  $\lambda$
is
\begin{align*}
G^{(k)}_\lambda &:= \bb_{\lambda^1} \mysquare \bb_{\lambda^2}\mysquare  \cdots \mysquare \bb_{\lambda^{d-1}} \mysquareb s_{\lambda^d} \, \in \Lambda \, .
\end{align*}
\end{definition}


\begin{theorem}
\label{t k split to k schur}
 Let  $\lambda \in \Par^k$ and $\lambda^{\rightarrow k} = (\lambda^1, \dots, \lambda^d)$ be the $k$-split of $\lambda$.
For each  $i \in [d]$, let $r_i = k-(\lambda^i)_1+1$, $U^i = (k-r_i+1)^{r_i}$, $\theta^i =  U^i/ \lambda^i$, and
$N_i \subset \ZZ_{\ge 1}$ be the interval of length  $r_i$ such that the restriction of $\lambda$ to positions
$N_i$ is $\lambda^i$.
Set $U = U^1 \cdots U^d$, the concatenation of the  $U^i$.
The $k$-Schur expansion of the $k$-split polynomial $G^{(k)}_{\lambda}$ is given by
\begin{align}
\label{et k split to k schur}
G^{(k)}_{\lambda} = \fs^{(k)}_{U}
\cdot \bigg( \sum_{T \in \SSYT_{\theta^d}(N_d)} \! \!\!\! u_{\creading(T)}\bigg)
\cdots
\bigg( \sum_{T \in \SSYT_{\theta^2}(N_{2})} \!\!\!\! u_{\creading(T)}\bigg)
\bigg( \sum_{T \in \SSYT_{\theta^1}(N_1)} \!\!\!\! u_{\creading(T)}\bigg).
\end{align}
\end{theorem}
The proof, given in Section \ref{s ksplit proofs}, is by induction on $d$ using Theorem \ref{t schur kschur}.

\subsection{$k$-Schur functions via $k$-split polynomials}
\label{ss kschur via ksplit}

By \cite[Property 29]{LMksplit}, the  $k$-split polynomials $\{G^{(k)}_\lambda \mid \lambda \in  \Par^k \}$ form a $\ZZ[t]$-basis of $\Lambda^k$.
We can thus write $\Lambda^k$ as the direct sum of its free  $\ZZ[t]$-submodules
\begin{align}
\label{e k split space 1}
\tilde{\Omega}^{k,a} &:= \Span_{\ZZ[t]} \{G^{(k)}_\lambda \mid \lambda \in  \Par^k, \, \lambda_1 = a\}
\end{align}
over $0 \le a \le k$; note that $\tilde{\Omega}^{k,0} := \Span_{\ZZ[t]}\{G^{(k)}_\emptyset \} $ and $G^{(k)}_\emptyset =1$.
Let $\pi^{k,d} : \Lambda^k \to \Lambda^k$ denote the projection with kernel $\bigoplus_{a > d} \tilde{\Omega}^{k,a}$ and which is the identity on
$\bigoplus_{a \le d} \tilde{\Omega}^{k,a}$.
A family of symmetric functions $\{\tilde{A}^{(k)}_\mu \mid \mu \in \Par^k \} \subset \Lambda$
were introduced and conjectured to be  $k$-Schur functions in \cite{LMksplit}.
They are defined inductively by
\begin{align}
\label{ed k split}
\tilde{A}^{(k)}_\mu = \begin{cases}
  \pi^{k,\mu_1} \big( \bb_{\mu_1} \mysquareb \tilde{A}^{(k)}_{(\mu_2, \dots, \mu_\ell)} \big) & \text{ if $\ell = \ell(\mu) > 0$},\\
1  & \text{ if $\mu = \emptyset$}.
\end{cases}
\end{align}

This  $k$-Schur candidate agrees with  the  $k$-Schur Catalan functions (proof in Section~\ref{s ksplit proofs}).

\begin{theorem}
\label{t k split eq catalan}
For all  $\mu \in \Par^k$, $\tilde{A}^{(k)}_\mu = \fs^{(k)}_{\mu} $.
\end{theorem}

Hence, by combining this with \cite[Theorem 2.4]{BMPS}, we have identified three of the $k$-Schur candidates  (Theorem \ref{c three defs agree}).
Properties of each are thus acquired by the others.
As one application, the following \emph{$k$-rectangle property} was proven for the $\tilde{A}^{(k)}_\mu$
\linebreak[3] \cite[Theorem 3]{LMirred}, so it is now established
for the strong tableau  $k$-Schur functions and  $k$-Schur Catalan functions.
A \emph{$k$-rectangle} is a partition of the form $(k-r+1)^r$ for $r \in [k]$.

\begin{corollary}
\label{c k rectangle}
Let  $U = (k-r+1)^r$ be a  $k$-rectangle and  $\mu \in \Par^k$.
Then
\[\bb_U \mysquareb s^{(k)}_{\mu} = \bb_U \mysquareb \fs^{(k)}_{\mu} = \bb_U \mysquareb \tilde{A}^{(k)}_{\mu}= t^d \tilde{A}^{(k)}_{U \sqcup \mu} = t^d \fs^{(k)}_{U \sqcup \mu} = t^d s^{(k)}_{U \sqcup \mu},\]
where $d$ is the number of boxes in the diagram of  $\mu$ in columns  $> r$,
and $U \sqcup \mu$ denotes the partition rearrangement of the parts of $U$ and $\mu$.
\end{corollary}

\section{Gromov-Witten invariants}
\label{s gromov witten}
Using a result of Lam-Shimozono \cite{LStoda},
we identify Gromov-Witten invariants with certain $k$-Catalan-Kostka coefficients at $t=1$ and apply
Theorem \ref{t schur kschur} to give a positive combinatorial formula for a class of these invariants (Theorem \ref{cor:gw2schurtimeskschur}).


We briefly introduce the quantum cohomology ring; further details can be found in \cite{FGP,Kim,RuanTian}.
Let $\Fl_{k+1}$ be the variety of complete flags in $\CC^{k+1}$.
Its cohomology ring  $H^*(\Fl_{k+1}) = H^*(\Fl_{k+1}, \CC)$ has a basis of Schubert classes  $\sigma_w$ indexed
by permutations $w\in \SS_{k+1}$.
The (small) \emph{quantum cohomology ring} $QH^* (\Fl_{k+1})$ is a commutative and associative algebra over
$\CC[{\mathbf q}]:=\CC[q_1,\ldots,q_{k}]$.
As a  $\CC[\mathbf{q}]$-module, $QH^*(\Fl_{k+1})= $ $\CC[{\mathbf q}]\otimes H^*(\Fl_{k+1})$
and thus has a $\CC[{\mathbf q}]$-basis
of Schubert classes $\{\sigma_w\}_{w \in \SS_{k+1}}$. The
multiplication~$*$ in  $QH^*(\Fl_{k+1})$ is determined by
\begin{equation}
\label{eq:gw}
  \sigma_u * \sigma_v = \sum_{w \in \SS_{k+1}} \sum_{\dd \in \ZZ_{\ge 0}^k} c_{u \hspace{.3mm} v}^{w,\dd }\, \mathbf{q}^{\mathbf d}\, \sigma_w,
\end{equation}
where the coefficients $c_{u \hspace{.3mm} v}^{w,\mathbf d} = \langle \sigma_u, \sigma_v, \sigma_{w_\circ w} \rangle_{\dd}$ are the 3-point \emph{Gromov-Witten invariants} of genus 0 and
$w_\circ$ is the longest element of $\SS_{k+1}$.
The  $\mathbf{d} = \mathbf{0}$ Gromov-Witten invariants $c_{u \hspace{.3mm} v}^{w, \mathbf{0}}$ are the Schubert structure constants.



Let $\bar{\Lambda}^k = \CC[h_1,\dots, h_{k}] = \CC \tsr_{\ZZ[t]} \Lambda^k$, the map  $\ZZ[t] \to \CC$
given by  $t \mapsto 1$.
The ring $\bar{\Lambda}^{k}$ has  $\CC$-basis $\{\sone_\mu^{(k)} \mid \mu \in \Par^k\}$
by \cite[Property 27]{LMktableau}, where the $\sone_\mu^{(k)}$ are the weak tableau $k$-Schur functions---see \S\ref{ss unifying}.
Note that $\sone^{(k)}_\mu = \fs^{(k)}_{\mu}(\mathbf{x};1)$ for all $\mu \in \Par^k$
by Theorem \ref{c three defs agree}.

Recall (\S\ref{ss kschur via ksplit}) that a $k$-rectangle is a partition of the form $R_i :=(i^{k+1-i})$ for  $i \in [k]$;
set $R_0=R_{k+1}:=\emptyset$.
A  $k$-bounded partition is {\it irreducible} if it
has at most $k-i$ parts of size~$i$ (equivalently, it contains no $k$-rectangle as a subsequence).

The next result features the localization
$\bar{\Lambda}^k[s_{R_1}^{-1},\ldots, s_{R_k}^{-1}]$
of  $\bar{\Lambda}^k$
and it will be useful to have notation for a basis.
Accordingly, let $\tpar^k$ denote the set of pairs  $(\nu, \mathbf{a})$ consisting of an
irreducible $k$-bounded partition  $\nu$ and a vector $\mathbf{a} = (a_1,\dots, a_k) \in  \ZZ^k$.
Also set $\R^\mathbf{a} := (\emptyset, \mathbf{a}) \in \tpar^k$.
For each $(\nu, \mathbf{a}) \in \tpar^k$,  define
\begin{align}
\label{e krect sone}
\sone^{(k)}_{(\nu, \mathbf{a})} := \sone^{(k)}_\nu s_{R_1}^{a_1} \cdots s_{R_k}^{a_k} \, \in \bar{\Lambda}^k[s_{R_1}^{-1},\ldots, s_{R_k}^{-1}].
\end{align}
For any  $\mu \in \Par^k$, there is a unique irreducible partition, denoted $\mu_{\downarrow}$,
obtained from $\mu$ by removing as many $k$-rectangles as possible;
we identify $\mu$ with the element $(\mu_\downarrow, \mathbf{a}) \in \tpar^k$,
where  $a_i$ is the number of rectangles  $R_i$ removed.
This identification makes the notation \eqref{e krect sone} consistent with
earlier usage
since the  $k$-rectangle property  (Corollary \ref{c k rectangle}) at  $t=1$ yields
$\sone^{(k)}_\mu = \sone^{(k)}_{\mu_\downarrow} s_{R_1}^{a_1} \cdots s_{R_k}^{a_k} = \sone^{(k)}_{(\mu_\downarrow, \mathbf{a})}$.
This computation also shows that $\bar{\Lambda}^k[s_{R_1}^{-1},\ldots, s_{R_k}^{-1}]$ has basis
$\{\sone^{(k)}_{(\nu, \mathbf{a})} \mid (\nu, \mathbf{a}) \in \tpar^k\}$.
The notation above conveniently handles some products:
$\sone^{(k)}_{(\nu, \mathbf{a})} \sone^{(k)}_{\R^\mathbf{b}} = \sone^{(k)}_{(\nu, \mathbf{a})} \sone^{(k)}_{(\emptyset, \mathbf{b})} = \sone^{(k)}_{(\nu, \mathbf{a}+\mathbf{b})}$.


For  $w = w_1 \cdots w_{k+1} \in \SS_{k+1}$ in one-line notation,
the \emph{inversion sequence} $\Inv(w) \in \ZZ_{\ge 0}^{k+1}$ is given by
$\Inv_i(w) = \big|\{ j>i : w_i > w_j \}\big|.$
Define an injection $\zeta:\SS_{k+1}\to {\rm Par}^k$ as follows:
the $i$-th column of $\zeta(w)$ is
\begin{equation}
\label{perm2part}
 \binom{k+1-i}{2}+\Inv_{i}(w_\circ w)
\end{equation}
for all $i\in [k]$.
Set $\theta(w) = \zeta(w)_\downarrow$.

\begin{example}
For  $k= 6$ and $w = 1246357$ in one-line notation, the conjugate of $\zeta(w)$ is
$(21,15,9,4,3,1)$ and $\theta(w)=(4,3,2)$.
\end{example}

For $w \in \SS_{k+1}$, the \emph{descent set} of  $w$ is  $D(w)=\{i :w_i>w_{i+1}\}$, and
the \emph{descent vector} of  $w$ is $\bD(w) := \sum_{i \in D(w)} \epsilon_i \in \ZZ_{\ge 0}^k$.
Lam and Shimozono \cite{LStoda} combine powerful results of Givental, Kim, Ginzburg, Kostant, and Peterson \cite{GiventalKim, Kim, Ginzburg, Kostant, PetersonNotes} to obtain
the following:
\begin{theorem}
\label{t quantum iso}
There is a ring isomorphism $\Phi : QH^*( \Fl_{k+1})[q_1^{-1},\ldots,q_k^{-1}] \to \bar{\Lambda}^k[s_{R_1}^{-1},\ldots, s_{R_k}^{-1}]$
which maps the Schubert classes  $\sigma_w$, $w \in \SS_{k+1}$, and  the $q_i$ as follows:
\begin{equation}
\label{eq:ls}
\Phi(\sigma_w)=\frac{\sone_{\theta(w)}^{(k)}}{\prod_{i\in D(w)} s_{R_i}} = \frac{\sone_{\theta(w)}^{(k)}}{ \sone^{(k)}_{\R^{\bD(w)}} }, \qquad \Phi(q_i) =  \frac{s_{R_{i-1}}s_{R_{i+1}}}{s_{R_i}^2}
\,.
\end{equation}
\end{theorem}
\begin{proof}
This is Theorem 1.1 and Proposition 5.1 of \cite{LStoda} except that we have used an alternative description of $\theta(w)$ (denoted  $\lambda(w)$ in \cite{LStoda}).  The two descriptions are reconciled in
Lemma \ref{le:pain in the ass} below.
\end{proof}

Let $c_{i}=s_{k+1-i}\cdots s_k \in \SS_{k+1}$.
Any $w\in \SS_{k+1}$ has a unique factorization $w=c_k^{m_0}c_{k-1}^{m_1}\cdots c_{1}^{m_{k-1}}$ with
 $0 \le m_i \le k-i$ for all $i = 0, \dots,  k-1$.
 A word in the positive integers is \emph{cyclically increasing} if some rotation of it is weakly increasing.

\begin{lemma}
\label{le:pain in the ass}
Let $w\in \SS_{k+1}$ and set $(I_1,\dots,I_{k+1})=\Inv(w_\circ w)$.
For  $i \in [k-1]$, the number of parts  $n_i$ of size $i$ in $\theta(w)$ has the following descriptions:
\begin{equation}
\label{eq:multiplicities}
n_i=\begin{cases}
I_i-I_{i+1}-1 & \text{ if } i\not\in D(w),\\
k-i+I_i-I_{i+1} & \text{otherwise},
\end{cases}
\end{equation}
and $n_i = m_i$ for $m_i$ defined by the factorization $w=c_k^{m_0}c_{k-1}^{m_1}\cdots c_{1}^{m_{k-1}}$ above.
This describes $\theta(w)$ completely since it is irreducible and thus only has parts of size  $\le k-1$.
\end{lemma}

\begin{proof}
Since a partition $\lambda$ has $\lambda_i'-\lambda_{i+1}'$ parts of size $i$,
there are $\tilde{n}_i := k-i+I_{i}-I_{i+1}$ parts of size $i$ in $\zeta(w)$.
If $i\not\in D(w)$ then $k-i +1 \ge I_i > I_{i+1}$,
implying that  $2 (k-i+1) > \tilde{n}_i > k-i$.
On the other hand, if $i\in D(w)$ then $I_i\leq I_{i+1}$,  implying
$k-i \ge \tilde{n}_i$.
Hence $\theta(w)= \zeta(w)_\downarrow$ is obtained from $\zeta(w)$ by removing
one copy of the rectangle $R_i$ for  $i \notin D(w)$. The formula \eqref{eq:multiplicities} follows.

Now consider the factorization $w=c_k^{m_0}c_{k-1}^{m_1}\cdots c_{1}^{m_{k-1}}$ as above.
Let  $i \in [k-1]$. The partial product
$c_k^{m_0}c_{k-1}^{m_1}\cdots c_{k-i+1}^{m_{i-1}}$ in one-line notation has the form $\hat{w} v$ where
$\hat{w}= w_1\cdots w_{i}$,  $v = v_1 \cdots v_{k+1-i}$, and $w_i v$ is cyclically increasing with
$v_a>v_{a+1}$ for $a\in \{0,1, \dots, k+1-i\}$
(for corner cases $a =0$ and  $a=k+1-i$,
set $v_0 :=w_1$ and $v_{k+2-i} := w_i$).
It follows that $I_i=a$.
Since right multiplication by $c_{k-i}^{m_i}$ cyclically shifts $v$
to the left $m_i$ positions,
$$
c_k^{m_0}c_{k-1}^{m_1}\cdots c_{k-i}^{m_{i}}
=
\begin{cases}
\hat w \, v_{m_i+1}<\cdots<v_a>v_{a+1}<\cdots < v_{k+1-i}<v_1<\cdots<v_{m_i} &
\text{ if } i\not\in D(w), \\
\hat{w} \, v_{m_i+1}<\cdots<v_{k+1-i}<v_{1}<\cdots<v_a>v_{a+1}<\cdots< v_{m_i}
&\text{ otherwise,}
\end{cases}
$$
and $v_{m_i}<v_{m_i+1}$.
Therefore, if $i\not\in D(w)$, $I_{i+1}=a-m_i-1$ implies $m_i=I_{i}-I_{i+1}-1=n_i$
and otherwise, $I_{i+1}=k-i-m_i+a$ implies $m_i=I_i-I_{i+1}+k-i=n_i$.
\end{proof}

For a vector  $\dd \in \ZZ^k$, define $\tilde{\dd} = \sum_{i \in [k]} d_i(\epsilon_{i-1} -2\epsilon_i + \epsilon_{i+1}) \in \ZZ^{k}$ where
 $\epsilon_0 = \epsilon_{k+1} := 0$.

\begin{theorem}
\label{the:gw2redklr}
Gromov-Witten invariants are  $k$-Catalan-Kostka coefficients at  $t=1$.
Precisely, for $u,v,w\in \SS_{k+1}$ and $\mathbf d\in \ZZ_{\ge 0}^{k}$, let $\mu=\theta(u)$ and $\nu=\theta(v)$,
$\Psi=\Delta^k(\mu)\uplus\Delta^k(\nu)$, and
$\lambda= (\theta(w), \tilde{\dd} + \bD(u) + \bD(v) - \bD(w)) \in \tpar^k$.  Then
$$
c_{u \hspace{.3mm} v}^{w,\dd}  =
K_{\lambda,\mu\nu}^{\Psi,k}(1)
\quad
$$
provided $\lambda$ is identified with an element of  $\Par^k$ and $c_{u \hspace{.3mm} v}^{w,\dd} = 0$ otherwise.
\end{theorem}
Note that  $\lambda \in \Par^k$ is equivalent to $\tilde{\dd} + \bD(u) + \bD(v) - \bD(w) \in \ZZ_{\ge 0}^k$.

\begin{proof}
By Theorem \ref{t quantum iso}, applying $\Phi$ to \eqref{eq:gw} gives
\begin{equation*}
\label{eq:with lambda denominators}
\frac{\sone_{\mu}^{(k)}}{ \sone_{\R^{\bD(u)}}^{(k)} } \frac{  \sone_{\nu}^{(k)}}{\sone_{\R^{\bD(v)}}^{(k)}} =
\sum_{w,\dd}c_{u \hspace{.3mm} v}^{w,\dd}\,\Phi(\mathbf{q}^\dd) \frac{\sone_{\theta(w)}^{(k)}}{\sone_{\R^{\bD(w)}}^{(k)}} =
\sum_{w,\dd}c_{u \hspace{.3mm} v}^{w,\dd}\, \frac{\sone_{\theta(w)}^{(k)}}{\sone_{\R^{\bD(w)}}^{(k)}} \prod_{i \in [k]} \Big(\frac{s_{R_{i-1}}s_{R_{i+1}}}{s_{R_i}^2}\Big)^{d_i}
= \sum_{w,\dd}c_{u \hspace{.3mm} v}^{w,\dd}\, \sone_{(\theta(w), \tilde{\dd}-\bD(w))}^{(k)}
\,,
\end{equation*}
where we have used the notation from \eqref{e krect sone}.
Clearing denominators we obtain
$$
\sum_{w,\dd}c_{u \hspace{.3mm} v}^{w,\dd}\,{\sone_{(\theta(w), \tilde{\dd} + \bD(u) + \bD(v) - \bD(w))}}^{(k)} = \sone_\mu^{(k)}\sone_\nu^{(k)}= H(\Psi;\mu\nu)(\mathbf x;1)
= \sum_{\lambda\in\Par^k} K^{\Psi,k}_{\lambda , \mu\nu}(1)\,\sone^{(k)}_\lambda
\, ,
$$
where the second equality is by Proposition \ref{p Catalan mult}
together with the fact $\sone^{(k)}_\eta(\mathbf{x}) = \fs^{(k)}_{\eta}(\mathbf{x};1)$ for all $\eta \in \Par^k$.
\end{proof}


\begin{lemma}
\label{le:inversions}
For $u \in \SS_{k+1}$ with only one descent at position $j$,
\[ \theta(u)'= \big(k+1-j-\Inv_1(u), \ldots,k+1-j-\Inv_j(u) \big). \]
\end{lemma}

\begin{proof}
%
The partition $\mu=\theta(u)$
has $n_x:=\mu_x'-\mu_{x+1}'$
parts of size $x$ and  $\ell(\mu') \le k-1$ since $\mu$ is irreducible.
Let $(I_1,\dots,I_{k+1})=\Inv(w_\circ u)$.
Since  $D(u) = \{j\}$,
\eqref{eq:multiplicities} gives
\begin{align*}
n_i=
\left\{ \!\!
\begin{array}{l@{\ =\ }ll}
I_i-I_{i+1}-1 &\Inv_{i+1}(u)-\Inv_i(u) & \text{ if } i < j, \\
k-i+I_i-I_{i+1} & k+1-j-\Inv_j(u) & \text{ if } i = j, \\
I_i-I_{i+1}-1 & 0 & \text{ if } i > j,
\end{array}\right.
\end{align*}
where we have used $I_x=k+1-x-\Inv_x(u)$ for all  $x$
and  $\Inv_x(u)=0$ for all  $x > j$.
Therefore $\mu_i'=0$ for all $i>j$ and
$\mu_j'=n_j= k+1-j-\Inv_j(u)$.
From there, $\mu_{j-1}'-\mu_{j}'= \Inv_{j}(u)-\Inv_{j-1}(u)$
implies $\mu_{j-1}'= k+1-j-\Inv_{j-1}(u)$, and iterating gives the result.
\end{proof}


Our main result on Gromov-Witten invariants is obtained by
transferring  our knowledge of the  $k$-Catalan-Kostka coefficients obtained in Theorem \ref{t schur kschur}.

\begin{theorem}
\label{cor:gw2schurtimeskschur}
Let  $u,v,w \in \SS_{k+1}$ and  $\dd \in \ZZ_{\ge 0}^k$.
Suppose $u$ has only one descent at position $j$
and  $v_{m+1}\cdots v_{k+1}$ is cyclically increasing, where $m$ is the maximum index such that $\Inv_1(u)=\cdots =\Inv_m(u)$.
Then
\begin{align}
\label{et gw}
c_{u \hspace{.3mm} v}^{w,\mathbf d}\;= \sum_{T \in \SSYT_{U/\theta(u)}([r])}
\  \sum_{\substack{S \in \SMT^k(\creading(T) \, ; U\theta(v)) \\ \inside(S) = \lambda}}  1 \,,
\end{align}
where $r= k+1-j-\Inv_1(u)$,  $U = R_{k+1-r}$, and
$\lambda= \big(\theta(w), \tilde{\dd} + \epsilon_j + \bD(v) - \bD(w)\big) \in \tpar^k$.
\end{theorem}
Note that since the sum requires  $\inside(S) = \lambda$,  $c_{u \hspace{.3mm} v}^{w,\mathbf d} = 0$ if $\lambda \notin \Par^k$.

\begin{proof}
Let $\mu=\theta(u)$ and $\nu=\theta(v)$.
Theorem~\ref{the:gw2redklr} in the case $\bD(u)=\epsilon_j$ gives
$c_{u \hspace{.3mm} v}^{w,\mathbf d}=K_{\lambda,\mu\nu}^{\Psi,k}(1)$
with $\Psi=\Delta^k(\mu)\uplus\Delta^k(\nu)$
and $\lambda$ as in the statement above.

The result then follows by computing $K_{\lambda,\mu\nu}^{\Psi,k}(1)$ using Theorem~\ref{t schur kschur}; to apply this theorem we need
to check (1) $r = \ell(\mu)$ and $\mu_1 \le k-r+1$, and (2) 
$\mu\nu$ is a partition.
Lemma~\ref{le:inversions} implies  $r = \ell(\mu)$ and
 $\mu_1 \le j = k-r+1-\Inv_1(u)$, giving  (1).
For (2), since $u$ has only one descent and $\Inv_1(u)=\cdots=\Inv_m(u)$,
Lemma~\ref{le:inversions} implies that the parts of $\mu$ are  $\ge m$.
It remains to show that $\nu$ has no part $> m$.
Letting $(I_1,\dots,I_{k+1})=\Inv(w_\circ v)$, we see from~\eqref{eq:multiplicities} that this
is equivalent to
(I) for $x>m$ with $x\in D(v)$, $I_x=I_{x+1}- (k-x)$, and (II) for $x>m$ with $x\notin D(v)$, $I_x=I_{x+1}+1$.
Condition (I) holds only if $I_x=0$ and $I_{x+1}=k-x$ since $k-x\geq I_{x+1}$.
One then checks that these conditions  are equivalent to $v_{m+1}\cdots v_{k+1}$ being cyclically increasing.
\end{proof}

\begin{example}
For $k=6$, $u=1246357$, and $v=1734562$, the expansion of $\sigma_u*\sigma_v$ can be
computed using Theorem~\ref{cor:gw2schurtimeskschur} since the only descent of $u$ is at position $j=4$,
$\Inv_1(u)=\Inv_2(u)=0$, and $v_3\cdots v_7$ is cyclically increasing.
We have $\theta(u) = 432$ and $\theta(v) = 211111$, so we can make use of Example \ref{ex schur times kschur}.  In particular, for fixed  $\lambda$
(which depends on  $w$ and  $\dd$ to be determined), we have already computed the right side of \eqref{et gw}.
For example, 
for $\lambda = 442111111$, we deduce from \eqref{e Schur times kSchur}
that the right side of \eqref{et gw} is 1. To see which  $w$ and  $\dd$ produce this  $\lambda$,
consider
\[442111111 = \lambda=\big(\theta(w), \tilde{\dd} + \epsilon_j + \bD(v) - \bD(w)\big)
= \big(\theta(w), \tilde{\dd} + \epsilon_4 + \epsilon_2 + \epsilon_6 - \bD(w)\big).\]
Thus $\theta(w) = \lambda_\downarrow = 442$, which implies by Lemma \ref{le:pain in the ass} that $w = 1245367$ up to left multiplication by some power of $c_k = s_1 \cdots s_k$.
For each possibility for $w$, we can solve for $\tilde{\dd}$ and then for $\dd$; if the answer lies in $\ZZ_{\ge 0}^{k}$, then this yields a term of $\sigma_u * \sigma_v$.
For example, with $w = 1245367$, this yields $\tilde{\dd} = \epsilon_1 - \epsilon_2 - \epsilon_6$ and $\dd = \epsilon_2 + \epsilon_3 + \epsilon_4 + \epsilon_5 + \epsilon_6$.
This accounts for the fourth term below.
The other terms are computed similarly.
$$
\sigma_u*\sigma_v =
\sigma_{1746352}+
\sigma_{2745361}+
\sigma_{2736451}+
q_2q_3q_4q_5q_6\big(\sigma_{1245367}+
\sigma_{1236457}+
\sigma_{2135467}\big)
\,.
$$
\end{example}

\section{The combinatorics of Catalan functions}
We establish several important lemmas
which express a Catalan function as the
sum of other Catalan functions with similar indexed root ideals.
We begin by introducing bounce graphs,
a natural combinatorial object arising in these computations. 

\subsection{Bounce graphs}
\label{ss bounce graphs}
These definitions are a review from \cite[\S5.1]{BMPS}.
We say that $\delta\in \Psi$ is a \emph{removable root of  $\Psi$} if  $\Psi \setminus \delta$ is a root ideal.

\begin{definition}
Fix a root ideal $\Psi\in\Delta^+_\ell$ and $x\in[\ell]$.
If there is a removable root  $(x,j)$ of  $\Psi$, then define $\down_\Psi(x) = j$; otherwise, $\down_\Psi(x)$ is undefined.
Similarly, if there is a removable root $(i,x)$ of  $\Psi$, then define $\upp_\Psi(x) = i$; otherwise, $\upp_\Psi(x)$ is undefined.
\end{definition}

\begin{definition}
\label{d row chain graph}
The \emph{bounce graph} of a root ideal  $\Psi \subset \Delta^+_\ell$ is the graph on the vertex set $[\ell]$
with edges $(r, \down_\Psi(r))$ for each $r\in [\ell]$ such that $\down_\Psi(r)$ is defined.
The bounce graph of  $\Psi$ is a disjoint union of paths called \emph{bounce paths of  $\Psi$}.

For each vertex $r \in [\ell]$,
distinguish $\chaindown_\Psi(r)$ (resp. $\chainup_\Psi(r)$) to be
the maximum (resp. minimum) element of the bounce path of  $\Psi$ containing  $r$.
For  $a,b \in [\ell]$ in the same bounce path of  $\Psi$ with  $a\le b$, we define
\[ \bpath_\Psi(a,b) = \big\{a, \down_\Psi(a), \down^2_\Psi(a), \dots, b\big\}, \]
i.e., the set of indices in this path lying between  $a$ and  $b$.
We also set  $\downpath_\Psi(r) = \bpath(r, \chaindown_\Psi(r))$  for any $r \in [\ell]$.
For $b = \down_\Psi^{m}(a)$, the \emph{bounce} from  $a$ to  $b$ is
\[ B_\Psi(a,b) := |\bpath_\Psi(a,b)|-1 = m.\]
\end{definition}

\begin{example}
Examples of $\bpath$, $\downpath$, and bounce for the root ideal $\Psi$ below:
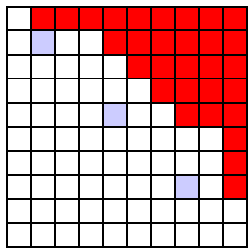
\ytableausetup{mathmode, boxsize=1.03em,centertableaux}
\[\begin{array}{cccccc}
{\tiny \begin{ytableau}
~ & *(red) & *(red)& *(red) & *(red)&*(red)&*(red)&*(red)&*(red)&*(red)\\
~ & *(blue!20) & & & *(red)&*(red)&*(red)&*(red)&*(red)&*(red)\\
~ & & & & &*(red)& *(red) & *(red) &*(red)&*(red)\\
~ & & & & & &*(red)&*(red)&*(red)&*(red)\\
~ & & & & *(blue!20) & & &*(red)&*(red)&*(red)\\
~ & & & & & & & & &*(red)\\
~ & & & & & & & & &*(red) \\
~ & & & & & & & *(blue!20) & & *(red)\\
~ & & & & & & & & & \\
~ & & & & & & & & &
\end{ytableau} } & & &
 {\tiny \begin{ytableau}
*(blue!20) & *(red) & *(red)& *(red) & *(red)&*(red)&*(red)&*(red)&*(red)&*(red)\\
~ & *(blue!20) & & & *(red)&*(red)&*(red)&*(red)&*(red)&*(red)\\
~ & & & & &*(red)& *(red) & *(red) &*(red)&*(red)\\
~ & & & & & &*(red)&*(red)&*(red)&*(red)\\
~ & & & & *(blue!20) & & &*(red)&*(red)&*(red)\\
~ & & & & & & & & &*(red)\\
~ & & & & & & & & &*(red) \\
~ & & & & & & & *(blue!20) & & *(red)\\
~ & & & & & & & & & \\
~ & & & & & & & & & *(blue!20)
\end{ytableau} }\\[16.4mm]
\bpath_{\Psi}(2,8) = \color{blue}\{2,5, 8\} & & & \downpath_{\Psi}(1) = \color{blue}\{1,2, 5,8,10\}
\end{array}\]
\vspace{-1mm}
\[\text{$B_\Psi(2, 8) = 2$, $B_\Psi(1,10) = 4$, $B_\Psi(3,6) = 1$, and $B_\Psi(3,3) = 0$}.\]
\end{example}

\begin{definition}
A root ideal $\Psi$ is said to have

\emph{a wall in rows $r,r+1,\dots, r+d$} if the rows $r, \dots, r+d$ of $\Psi$ have the same length,

\emph{a ceiling in columns $c,c+1$} if columns $c$ and $c+1$ of $\Psi$ have the same length, and

\emph{a mirror in rows} $r,r+1$ if $\Psi$ has removable roots $(r,c)$, $(r+1,c+1)$ for some $c > r+1$.
\end{definition}

\begin{example}
The root ideal $\Psi$ in the previous example has
a ceiling in columns  $2,3$, in columns $3,4$,  and in columns  $8,9$,
a wall in rows  $6,7,8$ and in rows  $9, 10$, and  a mirror in rows  $2,3$, in rows  $3, 4$, and in rows  $4, 5$.
\end{example}

\subsection{Mirror lemmas}

\label{ss mirror lemmas}
The following lemmas give sufficient conditions for a Catalan function to be zero or
for two Catalan functions to be equal.

\begin{lemma}[{\cite[Lemma 6.2]{BMPS}}]
\label{l toggle lemma zero}
Let $(\Psi,\eta)$ be an indexed root ideal of length $\ell$ and $z \in [\ell-1]$, and  suppose
\begin{align}
\label{el toggle lemma zero 1}
&\text{$\Psi$ has a ceiling in columns $z, z+1$;} \qquad \\[-.7mm]
\label{el toggle lemma zero 2}
&\text{$\Psi$ has a wall in rows $z,z+1$;}  \qquad \\[-.7mm]
\label{el toggle lemma zero 3}
&\text{$\eta_{z} = \eta_{z+1}-1$.} \qquad
\end{align}
Then $\HH(\Psi;\eta)=0$.
\end{lemma}

\begin{example}
By Lemma~\ref{l toggle lemma zero} with  $z=2$, the following Catalan function is zero:
\ytableausetup{mathmode, boxsize=1.1em,centertableaux}
\[
H\big(\big\{{\text{\small$(1,2), (1,3), (1,4), (1,5),(2,5),(3,5)$}} \big\}; 31211 \big)
=
{\scriptsize
\begin{ytableau}
   3    &  *(red) & *(red) &*(red)  & *(red)   \\
        &    1   &         &        & *(red)   \\
        &        &     2   &        & *(red) \\
        &        &         &    1   &          \\
        &        &         &        &    1
\end{ytableau}}
=0.
\]

Here and throughout the paper,
we depict the Catalan function  $H(\Psi; \gamma)$ of an indexed root ideal  $(\Psi,\gamma)$ of length  $\ell$
by the $\ell\times \ell$ grid of boxes, labeled by matrix-style coordinates,
with the boxes of $\Psi$ shaded and
the entries of $\gamma$ written along the diagonal.
\end{example}

\begin{lemma}[{\cite[Lemma 6.4]{BMPS}}]
\label{l little lemma many zero 2}
Let $(\Psi, \eta)$ be an indexed root ideal of length  $\ell$, and let $y,z,w$ be indices in the same bounce path of $\Psi$ with $1 \le y \le z \le w < \ell$, satisfying
\begin{align}
\label{el little lemma many zero 2 1}
&\text{$\Psi$ has a ceiling in columns $y, y+1$;}  \\[-.7mm]
\label{el little lemma many zero 2 2}
&\text{$\Psi$ has a mirror in rows $x, x+1$ for all $x \in \bpath_\Psi(y, w) \setminus \{w\} $;}  \\[-.7mm]
\label{el little lemma many zero 2 3}
&\text{$\Psi$ has a wall in rows $w,w+1$;}  \\[-.7mm]
\label{el little lemma many zero 2 4}
&\text{$\eta_{x} = \eta_{x+1}$ for all $x \in \bpath_\Psi(y,w) \setminus \{z\}$;}  \\[-.7mm]
\label{el little lemma many zero 2 5}
&\text{$\eta_{z} = \eta_{z+1} - 1$}.
\end{align}
Then $\HH(\Psi;\eta) = 0$.
\end{lemma}

Here is another useful variant, which is a simplified version of {\cite[Lemma 6.5]{BMPS}}.
\begin{lemma}
\label{l cascading toggle lemma}
Let $(\Psi,\eta)$ be an indexed root ideal of length $\ell$ and $z \in [\ell-1]$, and  suppose
\begin{align}
\label{el toggle lemma nonzero 1}
&\text{$\Psi$ has a ceiling in columns $z, z+1$;} \qquad \\[-.7mm]
\label{el toggle lemma nonzero 2}
&\text{$\Psi$ has a wall in rows $z,z+1$;}  \qquad \\[-.7mm]
\label{el toggle lemma nonzero 3}
&\text{$\eta_{z} = \eta_{z+1}$.} \qquad
\end{align}
If $\Psi$ has a removable root  $\delta$ in row $z+1$, then  $\HH(\Psi;\eta) = \HH(\Psi \setminus \delta ; \eta)$.
\end{lemma}

\subsection{Downpath lemmas}
\label{ss downpath lemmas}
We prove two lemmas which express a Catalan function as a sum over a downpath of similar Catalan functions.
We first recall \cite[Corollary 5.7]{BMPS}:

\begin{lemma}
\label{c inductive computation atom down}
Let $(\Psi,\eta)$ be an indexed root ideal of length  $\ell$
and  $p \in [\ell]$.
Then
\begin{align}
\HH(\Psi;\eta) = \sum_{z \in \downpath_\Psi(p)} t^{B_\Psi(p,z)}\HH(\Psi^z;\eta+ \epsilon_p - \epsilon_z),
\label{ec inductive computation atom down}
\end{align}
where $\Psi^z := \Psi \setminus \{(z, \down_\Psi(z))\}$ for $z \neq \chaindown_\Psi(p)$ and $\Psi^{\chaindown_\Psi(p)} := \Psi$.
\end{lemma}

\begin{lemma}
\label{l expand downpath}
Let $(\Psi, \eta)$ be an indexed root ideal of length  $\ell$, and $r \in [\ell]$,  $p \in [r]$.
Suppose
\begin{align}
\label{el expand 1}
&\text{$\eta_r \ge \eta_{r+1} \ge \dots \ge \eta_\ell$;}  \\[-.7mm]
\label{el expand 1b}
&\text{$\nr(\Psi)_i \ge r-i$ for  $i \in [r]$;}  \\[-.7mm]
\label{el expand 2}
&\text{$\style(\Psi, \eta)_i \le k$ for  $i \le r$ \ and \  $\style(\Psi, \eta)_i = \min(k, \ell-i+\eta_i)$ for $i>r$;}  \\[-.7mm]
\label{el expand 5}
&\text{$\style(\Psi, \eta)_p < k$;}   \\[-.7mm]
\label{el expand 4}
& \text{$\big(p< r$ and $\nr(\Psi)_p =  \nr(\Psi)_{p+1}\big)$  $\implies$  $\nr(\Psi)_{p+1} \le \nr(\Psi)_{p+2}$ and $\eta_p = \eta_{p+1} - 1$}.
\end{align}
Then
\begin{align}
\label{el expand downpath}
H(\Psi ; \eta )
&= \sum_{\{z \in \downpath_\Psi(p) \, \mid \, z = p \text{ or } \eta_z > \eta_{z+1}\}}
t^{B_\Psi(p,z)} H(\Psi^z ; \eta + \epsilon_p - \epsilon_z),
\end{align}
where $\Psi^z := \Psi \setminus \{(z, \down_\Psi(z))\}$ for $z \neq \chaindown_\Psi(p)$ and $\Psi^{\chaindown_\Psi(p)} := \Psi$.
\end{lemma}
In the corner case  $z = \ell$, we interpret the condition  $\eta_z > \eta_{z+1}$ by setting  $\eta_{\ell+1} = 0$
(though we still regard  $(\Psi, \eta)$ as an indexed root ideal of length  $\ell$, \emph{not}  $\ell+1$).

\begin{proof}
Lemma~\ref{c inductive computation atom down} gives
\begin{align*}
\HH(\Psi;\eta)   = \sum_{z \in \downpath_\Psi(p)} t^{B_\Psi(p,z)}\HH(\Psi^z ;\eta + \epsilon_p - \epsilon_z).
\end{align*}
If the sum contains a term with $z = \ell > p$ and  $\eta_\ell = 0$, then  $\HH(\Psi^z ;\eta + \epsilon_p - \epsilon_z)= 0$; this
follows directly from  Definition \ref{d HH gamma Psi}
since for  $\gamma \in \ZZ^\ell$, $s_\gamma =0$ whenever $\gamma_\ell < 0$ (see \eqref{ed s gamma}).

This given, the result is now obtained by applying Lemma~\ref{l little lemma many zero 2}
(with  $\Psi^z$ in place of  $\Psi$,  $\tilde{\eta} := \eta + \epsilon_p - \epsilon_z$ in place of  $\eta$,
 $w =z$, and  $y$ defined below)
to show that $\HH(\Psi^z ;\tilde{\eta}) = 0$ whenever  $p < z < \ell$ and $\eta_z = \eta_{z+1}$.
We verify the hypotheses \eqref{el little lemma many zero 2 1}--\eqref{el little lemma many zero 2 5}:
since  $z\ne p$, \eqref{el expand 1b} implies $z \ge \down_\Psi(p) > r$.
Hence
${\eta}_z = {\eta}_{z+1}$ and \eqref{el expand 2} imply $\Psi^z$ has a wall in rows $z, z+1$.
Let $y +1 = \chainup_{\Psi^z}(z+1)$.
Since  $\upp_{\Psi^z}(y+1)$ is undefined,  $\Psi^z$ has a ceiling in columns  $y, y+1$.

Since  $\down_{\Psi}(p)$ is defined,  $\nr(\Psi)_p \le \nr(\Psi)_{p+1}$.
If  $p=r$, then $\style(\Psi,\eta)_r < k$, \linebreak[4]
$\style(\Psi,\eta)_{r+1} = \min(k, \ell-i+\eta_{r+1})$, $\eta_r \ge \eta_{r+1}$, and  $z < \ell$ imply
$\nr(\Psi)_p < \nr(\Psi)_{p+1}$.
The two cases below therefore exhaust all possibilities and complete the proof.

\textbf{Case 1: $\nr(\Psi)_p < \nr(\Psi)_{p+1}$.}
In this case, $\upp_\Psi(q+1)$ is undefined, where $q := \down_\Psi(p)$.
It then follows from $\down_\Psi(p) > r$, \eqref{el expand 1}, and \eqref{el expand 2} that
 $y+1 \ge q+1$ and
\begin{align}
\label{e bpath plus 1}
\big\{x+1 \mid x\in \bpath_{\Psi^z}(y,z) \big\} = \bpath_{\Psi^z}(y+1,z+1).
\end{align}
Then \eqref{el expand 2} implies
$\tilde{\eta}_{x} = \tilde{\eta}_{x+1}$ for all $x \in \bpath_{\Psi^z}(y, z)\setminus \{z\}$. This verifies
\eqref{el little lemma many zero 2 2} and \eqref{el little lemma many zero 2 4}.

\textbf{Case 2:  $p < r$ and $\nr(\Psi)_p = \nr(\Psi)_{p+1}$.}
The same argument from case 1,
using in addi\-tion that
$\down_\Psi(p)+1= \down_\Psi(p+1)$ and $\tilde{\eta}_{p} = \tilde{\eta}_{p+1}$
by \eqref{el expand 4}, verifies \eqref{el little lemma many zero 2 2}~and~\eqref{el little lemma many zero 2 4}.
\end{proof}

%

The next result identifies a case where acting by a strong Pieri operator
on a  $k$-Schur function $\fs^{(k)}_{\eta}$
involves only  strong marked covers with height 1 ribbon components;
this is equivalent to the resulting sum of $k$-Schur functions 
being over only partitions
 contained in  $\eta$
(see Definition~7.9, the proof of Proposition~8.13, and Equation~9.6 of \cite{BMPS}).

\begin{corollary}
\label{c expand downpath}
Let  $\eta \in \Par^k_\ell$ and $\Phi = \Delta^k(\eta)$.
Let $r \in [\ell]$ and  $p \in [r]$.  Assume that $\nr(\Phi)_i \ge r-i$ for  $i \in [r]$
and that $p = r < \ell$ implies $\eta_r > \eta_{r+1}$.
Then
\begin{align}
\label{ec expand downpath}
\fs^{(k)}_{\eta} \cdot u_p = H(\Phi ; \eta - \epsilon_p)
&= \sum_{\{z \in \downpath_\Phi(p) \, \mid \, \eta -\epsilon_z \in \Par^k_\ell\}}
t^{B_\Phi(p,z)} \fs^{(k)}_{\eta - \epsilon_z}.
\end{align}
\end{corollary}
\begin{proof}
Lemma~\ref{l expand downpath} yields
\begin{align}
\label{ec expand downpath}
H(\Phi ; \eta - \epsilon_p)
&= \sum_{\{z \in \downpath_\Phi(p) \, \mid \, z= p \text{ or }\eta -\epsilon_z \in \Par^k_\ell\}}
t^{B_\Phi(p,z)} H(\Phi^z ; \eta - \epsilon_z),
\end{align}
where  $\Phi^z$ is as in \eqref{el expand downpath}.
For the terms with $\eta -\epsilon_z \in \Par^k_\ell$, we have $\Phi^z = \Delta^k(\eta-\epsilon_z)$
and hence $ H(\Phi^z ; \eta - \epsilon_z) = \fs^{(k)}_{\eta-\epsilon_z}$.
It remains to show the $z=p$ term is 0 when $\eta-\epsilon_p \not \in \Par^k_\ell$.
If  $p = \ell$,  then $\eta_p = 0$ and
$\HH(\Phi^p ;\eta  - \epsilon_p)= 0$ follows directly from  Definition \ref{d HH gamma Psi}
since for  $\gamma \in \ZZ^\ell$, $s_\gamma =0$ whenever $\gamma_\ell < 0$ (see \eqref{ed s gamma}).
If  $p < \ell$, then $\eta_p = \eta_{p+1}$.
So $p<r$ by the second assumption and then $\Phi$ has a ceiling in columns $p,p+1$ by the first.
Thus $H(\Phi^p ; \eta - \epsilon_p) = 0$ by Lemma~\ref{l toggle lemma zero}.
\end{proof}

\section{Schur times $k$-Schur: proof of Theorem \ref{t schur kschur}}
\label{s schur kschur proof}

In Theorem \ref{t schur kschur flag} below, we give a positive combinatorial formula for the $k$-Schur expansion of a large class of Catalan functions
which extrapolate between the Catalan functions corresponding to Schur times  $k$-Schur
of  Theorem \ref{t schur kschur} and single $k$-Schur functions.
This result strengthens Theorem \ref{t schur kschur} in a way that allows for a proof by induction on the size of the root ideal.
We begin by introducing a generalization of the tableaux $\SSYT_{R/\mu}([r])$ from Theorem \ref{t schur kschur}
needed  for the more general Theorem~\ref{t schur kschur flag}.


\subsection{Flagged tableaux}


A \emph{pseudopartition}  of length  $r$ is a sequence
\[\text{$\alpha = (\alpha_1, \dots, \alpha_r) \in \ZZ^r_{\ge 0}$ \quad satisfying  \, $\alpha_1 + r-1 \ge \alpha_2 + r-2 \ge \dots \ge \alpha_r$\,.}\]
For a pseudopartition  $\alpha$ of length  $r$  with $m= \max(\alpha)$, define
\[\diagram(\alpha)
:=
\big\{(i,j) \in [r] \times [m] \mid m - \alpha_i < j \le m \big\}.
\]

\begin{definition}
\label{d CTAB}
For a pseudopartition $\alpha$ of length $r$  and flagging $\mathbf{n} = (n_1,\dots,n_r) \in [r]^r$, define $\CTAB_\alpha(\mathbf{n})$ to be the set of
fillings of $\diagram(\alpha)$ with integers
which are (I) weakly increasing left to right along rows and
(II) strictly increasing down columns, and satisfy the following flag conditions:
(III) entries in row $i$ lie in $\{n_i,n_i+1,\dots, r\}$, and \linebreak[3] (IV) if   $i \in \{2,3,\dots, r\}$,
$(i-1,j) \in  \diagram(\alpha)$, and $(i,j) \not \in  \diagram(\alpha)$, then  the entry in box $(i-1,j)$ is less than $n_{i}$.
In the case  $\alpha = 0^r$ (and $\mathbf{n}$ any element of $[r]^r$),
$\CTAB_{\alpha}(\mathbf{n})$ consists of a single tableau with no boxes.

For $T \in \CTAB_\alpha(\mathbf{n})$, the \emph{column reading word} of $T$, denoted $\creading(T)$, is the word obtained by concatenating the columns of $T$ (reading each column bottom to top), starting with the leftmost column.

If $T \in \CTAB_\alpha(\mathbf{n})$, $x \in \ZZ$, and  $i \in [\ell(\alpha)]$, then ${\tiny \mytableau{x}}_{\, i} \sqcup T$
denotes the filling of  $\diagram(\alpha+\epsilon_i)$ 
obtained from  $T$ by adding a box with entry  $x$ to row  $i$ on the left.
\end{definition}

\begin{example}
With  $r=4$, $\alpha = 2121$, and $\mathbf{n} = 1224$, we have
\[\CTAB_\alpha(\mathbf{n}) =
\fontsize{6pt}{5pt}\selectfont
\Bigg\{ \
\tableau{1 & 1 \\\bl & 2 \\ 2 & 3 \\ \bl  & 4}\quad \ \
\tableau{1 & 1 \\\bl & 2 \\ 3 & 3 \\ \bl  & 4} \ \Bigg\}.
\]
See also Example \ref{ex proof schur kschur}.
If  $T$ is the tableau on the left, then ${\tiny \mytableau{2}}_{\, 3} \sqcup T = \fontsize{6pt}{5pt}\selectfont \tableau{\bl & 1 & 1 \\\bl& \bl & 2 \\ 2& 2 & 3 \\ \bl& \bl  & 4}$.
\end{example}

\begin{proposition}
\label{p flag decomposition tableau}
Let  $i \in [r]$. Let $\alpha$ and $\alpha^- = \alpha - \epsilon_{i}$ be pseudopartitions of length $r$ with  $\alpha_i \ge \alpha_{i-1}$.
Let $\mathbf{n} = (n_1,n_2,\dots,n_r) \in [r]^r$ and $\mathbf{n}^+ = \mathbf{n} + \epsilon_i$ be weakly increasing.
Then
\[\CTAB_\alpha(\mathbf{n}) = \CTAB_\alpha(\mathbf{n}^+) \sqcup \big\{ \, {\tiny \mytableau{n_i}}_{\, i} \sqcup U : U \in \CTAB_{\alpha^-}(\mathbf{n}) \big\}.\]
\end{proposition}
\begin{proof}
Consider  $T \in \CTAB_\alpha(\mathbf{n})$.
It either contains $n_i$ in the leftmost box in the $i$-th row or not.
If it contains  $n_i$, then removing this box yields
an element  $T^-$ of $\CTAB_{\alpha^-}(\mathbf{n})$ (note $T^-$ satisfies (IV) by column strictness of $T$), and if not, then $T \in \CTAB_\alpha(\mathbf{n}^+)$.
This establishes the inclusion  $\subset$.
To show the
inclusion $\supset$, we show that $\CTAB_\alpha(\mathbf{n}^+)$  and
$\big\{{\tiny \mytableau{n_i}}_{\, i} \sqcup U : U \in \CTAB_{\alpha^-}(\mathbf{n}) \big\}$
are each contained in  $\CTAB_\alpha(\mathbf{n})$.
The former comes down to the following:
$\alpha_i \ge \alpha_{i-1}$ implies that condition (IV) for $T \in \CTAB_\alpha(\mathbf{n}^+)$ yields condition (IV) for  $T$ considered as an element of $\CTAB_\alpha(\mathbf{n})$.
For the latter, let $T^- \in \CTAB_{\alpha^-}(\mathbf{n})$;
we must verify conditions (I)--(IV) for ${\tiny \mytableau{n_i}}_{\, i} \sqcup T^-$:
(I) and (III) are straightforward and (IV) follows from  $n_i < n_{i+1}$;
(II) follows from  $n_i < n_{i+1}$ and condition (IV) for $T^-$.
\end{proof}


\subsection{A generalization of Schur times  $k$-Schur}

For a sequence $\beta \in \ZZ^r$, define $\crop(\beta) \in \ZZ^r$ to be the partition given by
\[\crop(\beta)_i \, := \, \min\{\beta_j  \mid  j \in [i]\}.\]
For example, $\crop(45653213) = 44443211$.

\begin{theorem}
\label{t schur kschur flag}
Fix positive integers  $k,\ell$ and $r \in [\ell]$. Let  $(\Psi, \mu\nu)$ be an indexed root ideal of length  $\ell$ satisfying
\begin{align}
\label{et hyp 1}
&\text{$\mu$ is a pseudopartition of length  $r$;}\\
\label{et hyp 2}
&\text{$\nu \in \Par^k_{\ell-r}$ and $\mu_r \ge \nu_1$;}\\
\label{et hyp 4}
&\text{$\nr(\Psi)_i \ge r-i$ for $i \in [r]$;}\\
\label{et hyp 3}
&\text{$\style(\Psi, \mu\nu)_i \le k$ for  $i \le r$ \ and \ $\style(\Psi, \mu\nu)_i = \min(k, \ell-i+(\mu\nu)_i)$ for  $i > r$.}
\end{align}
Define $\zeta = \big(k-\nr(\Psi)_1, k-\nr(\Psi)_2, \dots, k-\nr(\Psi)_{r}\big)$,
$\lambda = \crop(\zeta)$,
and  $\alpha = \lambda-\mu$.
Let $\mathbf{n} = (n_1, \dots, n_r) \in [r]^r$ be the sequence given by $n_i = i - (\zeta_i- \lambda_i)$.
Also assume
\begin{align}
\label{et hyp 5}
 \text{$\alpha_i \ge 0$ for all  $i \in [r]$. \hspace{1.4cm} }
\end{align}
Then
\begin{align}
\label{et schur kschur flag}
H(\Psi;  \mu \nu) = \sum_{T \in \CTAB_\alpha(\mathbf{n})} \fs^{(k)}_{\lambda \nu} \cdot  u_{\creading(T)}\, .
\end{align}
\end{theorem}

Note that  $\zeta \in \ZZ^r_{\ge 0}$ by \eqref{et hyp 3}
and then $\lambda\nu\in \Par^k_\ell$ since  $\lambda_1 = \zeta_1 \le  k$
and  $\lambda_r \ge \mu_r \ge \nu_1$.
Also  $\mathbf{n} \in [r]^r$ as required by Definition \ref{d CTAB} since  $n_1 = 1 - (\zeta_1 - \lambda_1)= 1$
and  $\mathbf{n}$ is weakly increasing by Proposition \ref{p schur kschur flag helper} (e) below.

Let us now see how Theorem \ref{t schur kschur} is obtained as a special case of Theorem \ref{t schur kschur flag}:
let  $\mu \in \Par^{k-r+1}_r$
and  $\nu \in \Par^k$ be as in Theorem \ref{t schur kschur};
set  $\ell = \max(\ell(\mu\nu), r)$ so that  $\nu \in \Par^k_{\ell-r}$ and  $\mu\nu \in \Par^{k-r+1}_{\ell}$
(note that we allow  $\nu = \emptyset$).
Set $\Psi = \varnothing_r \uplus \Delta^k(\nu)$.
With this input to Theorem~\ref{t schur kschur flag}, we have
$\lambda = (k-r+1)^r$
and  $\mathbf{n} = (1, \dots, 1)$.
The latter implies $\CTAB_\alpha(\mathbf{n}) = \SSYT_{\lambda/\mu}([r])$,
which matches the right side of \eqref{et schur kschur} with the right side of \eqref{et schur kschur flag};
the left sides match by Proposition \ref{p Catalan box times} and \eqref{e Catalan op eq fun}.
The hypotheses of Theorem \ref{t schur kschur flag} are satisfied:
\eqref{et hyp 1}--\eqref{et hyp 4} are clear,
\eqref{et hyp 3} holds by definition of  $\Psi$ and  $\mu_1 \le k-r+1$, and
 \eqref{et hyp 5} holds by $\alpha_i = \lambda_i - \mu_i \ge k-r+1-\mu_1 \ge 0$.

In the next proposition, we handle several technical preliminaries
for the proof of \break Theorem~\ref{t schur kschur flag}.
For statement (a), we define the root ideal $\Delta^k(\gamma)$ for pseudopartitions $\gamma \in \ZZ^\ell_{\le k}$
just as in Definition \ref{d0 catalan kschur}.

\begin{proposition}
\label{p schur kschur flag helper}
Maintain the notation of Theorem \ref{t schur kschur flag}.
Set $\Phi = \Delta^k(\lambda\nu)$.
\begin{list}{\emph{(\alph{ctr})}} {\usecounter{ctr} \setlength{\itemsep}{1pt} \setlength{\topsep}{2pt}}
\item $\Delta^k(\lambda\nu) = \Phi \subset \Psi = \Delta^k(\zeta\nu)$.
\item For $j \in [r-1]$, \  $\zeta_j < \zeta_{j+1}$  $\iff$  \hspace{-1mm} $\Psi$ has a wall in rows $j,j+1$.
\item For $j \in [r-1]$,  $\alpha_j > \alpha_{j+1}$ and  $\zeta_j \le \zeta_{j+1}$ imply $\mu_j = \mu_{j+1}-1$ and $\alpha_j = \alpha_{j+1}+1$.
\item $(r, \down_\Psi(r))$ is a removable root of  $\Psi$ provided  $\ell > r+k$.
\item The sequence $\mathbf{n}$ is weakly increasing.
\item $\lambda_{n_p} = \lambda_{n_p+1} = \dots = \lambda_{p}$.
\item $\downpath_\Phi(n_p)\setminus \{n_p\} = \downpath_\Psi(p) \setminus \{p\}$ provided  $\ell > r+k$.
\end{list}
\end{proposition}

\begin{proof}
We first prove (a).  We have $\Delta^k(\zeta\nu) = \Psi$ by \eqref{et hyp 3} and the definitions of  $\zeta$ and  $\Delta^k$. By definition of  $\crop$,  $\zeta_i \ge \lambda_i$ for all $i \in [r]$.
Then by definition of  $\Delta^k$, \break
$\nr(\Phi)_i  = k-\lambda_i \ge k- \zeta_i = \nr(\Psi)_i$ for $i \in [r]$.
Hence $\Phi \subset \Psi$.
Statement (b) is clear from the fact that $\Psi$ is a root ideal.
For (c), $\zeta_j \le \zeta_{j+1}$ and the definition of  $\crop$ imply  $\lambda_j = \lambda_{j+1}$.
Then $\alpha_j > \alpha_{j+1}$ and  $\mu = \lambda-\alpha$ yield $\mu_j < \mu_{j+1}$. Since $\mu$ is a pseudopartition, we obtain
$\mu_j = \mu_{j+1}-1$ and $\alpha_j = \alpha_{j+1}+1$.
Statement  (d) follows from  $\mu_r \ge \nu_1$ and $\style(\Psi, \mu\nu)_r \le k= \style(\Psi, \mu\nu)_{r+1}$, where the latter holds by
 \eqref{et hyp 3} and  $\ell > r+k$.

Let  $i \in [r-1]$.  Since  $\nr(\Psi)_{i+1} \ge \nr(\Psi)_i - 1$,
we have  $\zeta_{i+1} - \zeta_i \le 1$ and hence
\[ \zeta_{i+1} - \lambda_{i+1}-(\zeta_i - \lambda_i)
  =
  \begin{cases}
  -(\zeta_i - \lambda_i) & \text{ if  $\zeta_{i+1} = \min\{\zeta_j \mid j \in [i+1] \} = \lambda_{i+1}$} \\
  \zeta_{i+1} - \zeta_i & \text{ if  $\zeta_{i+1} > \min\{\zeta_j \mid j \in [i]\} = \lambda_i$}
  \end{cases}
\]
is also $\le 1$. Statement (e) follows.

Let  $a \in [p]$ be any index such that  $\zeta_a = \lambda_p := \min_{j \in [p]}\{\zeta_j\}$.  Then  $\lambda$ is constant on the interval  $\{a,a+1,\dots, p\}$.
Since $\zeta_{p+1} - \zeta_p \le 1$ and  $\zeta_p - \zeta_a \le p-a$, there holds
\[n_p = p - (\zeta_p- \lambda_p) = p - \zeta_p + \zeta_a \ge a. \]
Statement (f) follows.

Since $\Phi$ and $\Psi$ agree in rows $>r$, statement (g) follows from the computation
\[\down_\Phi(n_p)-1 =  n_p +\nr(\Phi)_{n_p}= p-(\zeta_p-\lambda_p) + k- \lambda_{n_p}  = p + k-\zeta_p =  p + \nr(\Psi)_p = \down_\Psi(p)-1, \]
where the third equality is by (f).
\end{proof}

We also need the following result to adjust trailing zeros of  $\nu$ in the proof below.
\begin{proposition}[{\cite[Proposition 4.12]{BMPS}}]
\label{p trailing zeros}
If $(\Psi,\gamma)$ is an indexed root ideal
of length  $\ell$ with  $\gamma_\ell = 0$,  then $H(\Psi;\gamma) = H\big(\Psi \cap \Delta^+_{\ell-1} ;(\gamma_1,\dots, \gamma_{\ell-1})\big)$.
 \end{proposition}

\begin{proof}[Proof of Theorem \ref{t schur kschur flag}]
The proof is by
induction, beginning with several reductions,
and then addressing the left side of \eqref{et schur kschur flag} using
Lemma~\ref{l expand downpath}
(Step 1).  The terms in the sum are evaluated in Steps 2 and 3.
Then in Step 4, the right side of \eqref{et schur kschur flag} is evaluated using Corollary~\ref{c expand downpath} and matched with the result from Step 3.
Steps 2--4  split into two cases.

\noindent \textbf{Step 0: Base case and reductions.}

First, we can reduce to the case $\ell > r+k$ since adding extra zeros to $\nu$ does not change \eqref{et schur kschur flag}:
precisely, if  $\ell \le r+k$ and we let  $\tilde{\nu} = (\nu, 0^{r+k+1-\ell})$
and $\tilde{\Psi} \subset \Delta^+_{r+k+1}$ be the root ideal with  $\nr(\tilde{\Psi})_i = \nr(\Psi)_i$
for  $i \le r$ and $\nr(\tilde{\Psi})_i = \min(k-(\mu\tilde{\nu})_i, r+k+1-i)$ for
 $i> r$,  then
\begin{align*}
H(\Psi;\mu\nu) = H(\tilde{\Psi}; \mu\tilde{\nu})
=\sum_{T \in \CTAB_\alpha(\mathbf{n})} \fs^{(k)}_{\lambda\tilde{\nu}} \cdot  u_{\creading(T)}
=\sum_{T \in \CTAB_\alpha(\mathbf{n})} \fs^{(k)}_{\lambda\nu} \cdot  u_{\creading(T)}.
\end{align*}
The first equality is by Proposition \ref{p trailing zeros},
noting that $\tilde{\Psi} \cap \Delta^+_\ell = \Psi$ follows from \eqref{et hyp 3};
 the second is by the result in the $\ell > r+k$ case;
the third holds by
$\fs^{(k)}_{\lambda\nu}= \fs^{(k)}_{(\lambda\nu,0)} = \fs^{(k)}_{(\lambda\nu,0,0)} = \cdots$
(by Proposition \ref{p trailing zeros} again)
and the definition \eqref{e strong Pieri operator}
of the strong Pieri operators.

From now on we assume $\ell > r+k$.
We proceed by induction on $|\Psi| +  |\alpha|$.
By Proposition~\ref{p schur kschur flag helper} (a), the base case is $\alpha = 0^r$ and $\Psi = \Delta^k(\lambda\nu)$.
The desired \eqref{et schur kschur flag} holds in this case since $H(\Psi;\mu\nu) = H(\Delta^k(\lambda\nu) ; \lambda \nu) = \fs^{(k)}_{\lambda \nu}$
and $\CTAB_\alpha(\mathbf{n})$ consists of a single tableau with no boxes.

We now establish that for any  $i \in [r-1]$,
\begin{align}
\label{e both sides zero}
\text{$\alpha_i > \alpha_{i+1}$ and $\Psi$ has a wall in rows  $i, i+1$}  \ \ \implies \ \ \text{both sides of \eqref{et schur kschur flag} are 0}.
\end{align}
Accordingly, suppose the left side of \eqref{e both sides zero} is true.
We have
\begin{align}
\label{e wall etc}
\text{$\Psi$ has a wall in rows $i, i+1$}
\ \implies \
\zeta_i = \zeta_{i+1}-1
\ \implies \ \lambda_i = \lambda_{i+1}.
\end{align}
This gives $\mu_i = \mu_{i+1}-1$ by Proposition \ref{p schur kschur flag helper} (c) as well as
\[
n_i = i - (\zeta_i- \lambda_i) = i+1 - (\zeta_{i+1}- \lambda_{i+1}) = n_{i+1}.
\]
Therefore
the right side of \eqref{et schur kschur flag} is 0 since  $\CTAB_\alpha(\mathbf{n}) = \varnothing$ whenever $\alpha_i > \alpha_{i+1}$ and  $n_i = n_{i+1}$
(by condition (IV) in Definition \ref{d CTAB}) and the left side is 0 by Lemma~\ref{l toggle lemma zero} (with  $z = i$).
Hence \eqref{e both sides zero} is established.

We next establish that for any $i \in [r-1]$,
\begin{align}
\label{e reduction 2}
 \text{$\Psi$ has a wall in rows  $i, i+1$ \, and \, } \alpha_{i} = \max\{\alpha_i, \dots, \alpha_r\} \ \implies \ \eqref{et schur kschur flag} \text{ holds}.
\end{align}
Let  $j \in [i+1, r]$ be the largest number
such that  $\Psi$ has a wall in rows  $i, i+1, \dots, j$.
By \eqref{e both sides zero} and  $\alpha_{i} = \max\{\alpha_i, \dots, \alpha_r\}$,
we can assume  $\alpha_i = \dots = \alpha_j$.
By Proposition \ref{p schur kschur flag helper} (b), and
the definition of  $\crop$,  $\lambda_{i} = \dots = \lambda_{j}$.
Then $\mu_{j-1} = \mu_{j}$  (using $\mu = \lambda - \alpha$).
By definition of  $j$ (and by Proposition \ref{p schur kschur flag helper} (d) if  $j = r$),
$\delta :=  (j, \down_\Psi(j))$ is a removable root of  $\Psi$.
The last two sentences show that we can apply Lemma~\ref{l cascading toggle lemma} (with $z=j-1$) to obtain
\begin{align*}
H(\Psi; \mu\nu) &= H(\Psi \setminus \delta; \mu\nu) = \sum_{T \in \CTAB_\alpha(\mathbf{n}^+)} \fs^{(k)}_{\lambda\nu} \cdot  u_{\creading(T)}
= \sum_{T \in \CTAB_\alpha(\mathbf{n})} \fs^{(k)}_{\lambda\nu} \cdot  u_{\creading(T)}\, .
\end{align*}
The second equality is by the inductive hypothesis with data $\zeta - \epsilon_{j}$ in place of  $\zeta$, $\mathbf{n}^+ = \mathbf{n}+\epsilon_j$
in place of  $\mathbf{n}$,
and  $\mu, \lambda$, $\alpha$ unchanged (the hypotheses \eqref{et hyp 1}--\eqref{et hyp 4} and \eqref{et hyp 5} are clear while \eqref{et hyp 3} follows from
 $\style(\Psi, \mu\nu)_j = \style(\Psi, \mu\nu)_{j-1}-1 < k$).
The third equality is by $\CTAB_\alpha(\mathbf{n}^+) = \CTAB_\alpha(\mathbf{n})$,
which follows from $\alpha_{j-1} = \alpha_j$ and
\[n_{j-1} = j-1-(\zeta_{j-1}-\lambda_{j-1})= j-1-(\zeta_{j}-1-\lambda_{j}) =  n_j.\]

Next suppose $\alpha = 0^r$.  Since we have already handled the base case, we may assume
$\Delta^k(\zeta\nu) = \Psi \supsetneq \Delta^k(\lambda \nu) = \Delta^k(\mu \nu)$ (see Proposition \ref{p schur kschur flag helper} (a)).
Then $\lambda \ne \zeta$, so
there exists an $i \in [r-1]$ such that $\zeta_i < \zeta_{i+1}$.  Hence
$\Psi$ has a wall in rows $i, i+1$, so we are done by \eqref{e reduction 2}.

\noindent \textbf{Step 1: expand the left side of \eqref{et schur kschur flag}.}

We can now assume  $|\alpha| > 0$.  Set
\[p = \max\{j \in [r] \mid  \alpha_j = \max(\alpha) \}.\]
If  $p = r$, then $\delta := (p, \down_\Psi(p))$ is a removable root of  $\Psi$ by Proposition \ref{p schur kschur flag helper} (d).
If  $p < r$ and $\Psi$ has a wall in rows  $p, p+1$, then we are done by  \eqref{e both sides zero}.
So from now on we may assume $\delta$ is a removable root of  $\Psi$.
In addition, by \eqref{e reduction 2}, we may assume that if $p+1 <r$ and $\alpha_{p+1} = \max\{\alpha_{p+1}, \dots, \alpha_r\}$, then
$\nr(\Psi)_{p+1} \le \nr(\Psi)_{p+2}$.

Lemma~\ref{l expand downpath} yields
\begin{align}
\label{e schur kschur two terms}
H(\Psi; \mu\nu) &= \sum_{\{ z \in \downpath_\Psi(p) \, \mid \,  z = p \text{ or } (\mu\nu)_z > (\mu\nu)_{z+1} \}} \!\!\!\! t^{B_\Psi(p,z)}\HH(\Psi^z ;\mu \nu + \epsilon_{p}-\epsilon_z),
\end{align}
where $\Psi^z := \Psi \setminus \{(z,\down_\Psi(z))\}$ for $z \ne \chaindown_\Psi(p)$ and $\Psi^{\chaindown_\Psi(p)} := \Psi$.
We must verify the hypotheses of the lemma:
\eqref{el expand 1} is clear from \eqref{et hyp 1} and \eqref{et hyp 2}, \eqref{el expand 1b} is immediate from \eqref{et hyp 4},
\eqref{el expand 2} is immediate from
 \eqref{et hyp 3}, and
\eqref{el expand 5} follows from
\begin{align}
\label{e style less k}
\style(\Psi, \mu\nu)_p = \nr(\Psi)_p + \mu_p \le k- \lambda_p + \mu_p < k,
\end{align}
where the first inequality is by  $\lambda_p \le \zeta_p$ and the second
is by $\lambda_p - \mu_p = \alpha_p > 0$.
Finally, we verify \eqref{el expand 4}: suppose $p < r$ and
$\nr(\Psi)_p = \nr(\Psi)_{p+1}$.
By Proposition \ref{p schur kschur flag helper} (c), $\mu_{p} = \mu_{p+1}-1$ and $\alpha_p = \alpha_{p+1}+1$.
The latter implies $\alpha_{p+1} = \max\{\alpha_{p+1}, \dots, \alpha_r\}$.
The conclusion
$\nr(\Psi)_{p+1} \le \nr(\Psi)_{p+2}$
thus follows
from the previous paragraph if  $p+1 < r$ and from Proposition \ref{p schur kschur flag helper} (d) if  $p + 1 = r$.


\noindent \textbf{Step 2: apply the inductive hypothesis to the terms arising in Step 1.}

We proceed by applying the inductive hypothesis to any term from the sum in \eqref{e schur kschur two terms} with  $z > p$
(with $\mu^+ = \mu + \epsilon_p$  in place of  $\mu$, 
 $\alpha^- = \alpha - \epsilon_p$ in place of  $\alpha$, $\Psi^z$ in place of  $\Psi$,
 and  $\lambda, \zeta, \mathbf{n}$ unchanged);
the hypotheses are readily verified:
\eqref{et hyp 2} and \eqref{et hyp 4} are clear, \eqref{et hyp 5} holds since $\alpha_p >  0$,
and \eqref{et hyp 3} holds by \eqref{e style less k}  and the fact that
$\Psi^z$ and  $\Delta^k(\mu\nu+\epsilon_p-\epsilon_z)$ agree in rows  $ > r$.
For \eqref{et hyp 1},  we have
\begin{align*}
 \lambda_{p-1} \ge \lambda_p \, \text{ and } \, \alpha_{p-1} \le \alpha_p \, \text{ and } \, \mu = \lambda-\alpha  \, \, \implies  \, \, \mu_{p-1} \ge \mu_p \, \,  \implies \, \,  \mu^+ \text{ is a pseudopartition}.
\end{align*}
Hence by the inductive hypothesis,
\begin{align}
\label{e case 2 LHS mu plus}
\HH(\Psi^z ;\mu \nu + \epsilon_{p}-\epsilon_z) =
\fs^{(k)}_{\lambda \nu - \epsilon_z} \cdot \sum_{T \in \CTAB_{\alpha^-}(\mathbf{n})} \!\! u_{\creading(T)}.
\end{align}

We next consider the  $z=p$ term $\HH(\Psi \setminus \delta ;\mu \nu)$ of \eqref{e schur kschur two terms}.
Set
\[\zeta^- = \zeta- \epsilon_p \ \ \text{ and } \ \ \lambda^- = \crop(\zeta^-),\]
 which is part of the new data for this term
(the new  $\alpha$ and $\mathbf{n}$ will be addressed later).
The remainder of the proof separates into the cases
$\lambda \ne \lambda^-$ (case 1) and $\lambda = \lambda^-$ (case 2);
case 1 breaks into the two subcases below.

\textbf{Case 1A: $\lambda^-$ and  $\lambda$ differ in more than one position.}

It follows from the definition of  $\crop$ that $\lambda_p = \lambda_{p+1}$,
$\lambda^-_p = \lambda^-_{p+1} = \lambda_p-1$,
and  $\zeta^-_p = \lambda^-_p$.
This gives $\zeta^-_{p+1} = \zeta_{p+1} \ge  \lambda_{p+1} > \lambda^-_p = \zeta^-_p$, and thus
 $\Psi \setminus \delta$ has a wall in rows $p, p+1$ by  Proposition \ref{p schur kschur flag helper} (b).
Moreover, $\lambda_p = \lambda_{p+1}$,  $\alpha_p > \alpha_{p+1}$, and  $\mu$ a pseudopartition  give
$\mu_p = \mu_{p+1}-1$.
Hence  by Lemma~\ref{l toggle lemma zero} (with $z=p$),
\begin{align}
\label{e case 1 LHS nonroot}
\HH(\Psi \setminus \delta ;\mu \nu) = 0.
\end{align}


\textbf{Case 1B: $\lambda^- = \lambda-\epsilon_p$.}

We apply the inductive hypothesis to the term $\HH(\Psi \setminus \delta ;\mu \nu)$ with data  $\zeta^-$,  $\lambda^-$ as mentioned above,
$\alpha^- = \lambda^- - \mu = \alpha-\epsilon_p$, and $\mathbf{n}$ unchanged since
\[ \zeta^- - \crop(\zeta^-) = \zeta -\epsilon_p - \lambda^-  = \zeta-\epsilon_p - (\lambda-\epsilon_p) =  \zeta - \lambda. \]
The hypotheses are satisfied: \eqref{et hyp 1}, \eqref{et hyp 2}, and \eqref{et hyp 4} are clear, \eqref{et hyp 3} follows from \eqref{e style less k}, and \eqref{et hyp 5}
holds since $\alpha_p >  0$.
The inductive hypothesis thus gives
\begin{align}
\label{e case 2 LHS nonroot}
\HH(\Psi \setminus \delta; \mu\nu )  = \fs^{(k)}_{\lambda^- \nu} \cdot \sum_{T \in \CTAB_{\alpha^-}(\mathbf{n})} \!\! u_{\creading(T)}.
\end{align}

\textbf{Case 2: $\lambda^- = \lambda$.}

We apply the inductive hypothesis to the term $\HH(\Psi \setminus \delta ;\mu \nu)$
with data $\zeta^-$,  $\lambda^-$ as mentioned above, $\alpha$ unchanged, and $\mathbf{n}^+ := \mathbf{n}+\epsilon_p$, which follows from
\[ \zeta^- - \lambda^- = \zeta-\epsilon_p - \lambda^-  = (\zeta - \lambda) -\epsilon_p. \]
The hypotheses are satisfied: \eqref{et hyp 1}, \eqref{et hyp 2}, \eqref{et hyp 4}, and \eqref{et hyp 5} are clear, and \eqref{et hyp 3} follows from \eqref{e style less k}.
Hence by the inductive hypothesis,
\begin{align}
\label{e case 3 LHS nonroot}
\HH(\Psi \setminus \delta; \mu\nu)  = \fs^{(k)}_{\lambda \nu} \cdot \sum_{T \in \CTAB_\alpha(\mathbf{n}^+)}  u_{\creading(T)} \, .
\end{align}

\noindent \textbf{Step 3 (case 1): assemble the terms from step 2.}

In this case ($\lambda \ne \lambda^-$),
combining \eqref{e schur kschur two terms}, \eqref{e case 2 LHS mu plus}, \eqref{e case 1 LHS nonroot}, and \eqref{e case 2 LHS nonroot} yields
\begin{align}
\label{e case 2 LHS final}
H(\Psi; \mu\nu) &=
\sum_{\{z \in \downpath_\Psi(p) \, \mid \, \lambda\nu - \epsilon_z \in \Par^k_\ell \}}
\!\!\!\! t^{B_\Psi(p,z)} \Bigg( \fs^{(k)}_{\lambda \nu - \epsilon_z}   \cdot
\sum_{T \in \CTAB_{\alpha^-}(\mathbf{n})}  \!\! u_{\creading(T)} \Bigg).
\end{align}
Note that this holds in both subcases of case 1:
in case 1A,  $\lambda_p =\lambda_{p+1}$ and so there is no $z=p$ term in this sum as required by \eqref{e case 1 LHS nonroot};
in case 1B, $\lambda^- \nu = \lambda \nu - \epsilon_p$ is a partition, and
therefore this sum has a term for $z = p$ and it agrees with the right side of \eqref{e case 2 LHS nonroot}.

\noindent \textbf{Step 4 (case 1): evaluate the right side of \eqref{et schur kschur flag} to match it with \eqref{e case 2 LHS final}.}

It follows from the definition of  $\crop$ and $\lambda^- \ne \lambda$ that
$\lambda_p := \min_{j \in [p]}(\zeta_j) = \zeta_p$; in addition, if  $p < r$, then  $\zeta_p \ge \zeta_{p+1}$ (since  $\delta$ is removable) and therefore
$\lambda_{p+1} := \min_{j \in [p+1]}(\zeta_j) = \zeta_{p+1}$.
Hence ${n}_p = p$ and, if  $p < r$, ${n}_{p+1} = p+1$.
Then by Proposition~\ref{p flag decomposition tableau}, $\CTAB_\alpha(\mathbf{n}) =
\big\{{\tiny \mytableau{n_p}}_{\, p} \sqcup T : T \in \CTAB_{\alpha^-}(\mathbf{n})\big\}$
(condition (IV) of Definition \ref{d CTAB} implies $\CTAB_\alpha(\mathbf{n}+\epsilon_p) = \varnothing$).
The consequence for column reading words is thus
\begin{align}
\label{e colwords split}
\sum_{T \in \CTAB_\alpha(\mathbf{n})} u_{\creading(T)} &= u_{n_p}\sum_{T \in \CTAB_{\alpha^-}(\mathbf{n})} u_{\creading(T)}.
\end{align}

We need to compute the action of the operator in \eqref{e colwords split}
on $\fs^{(k)}_{\lambda\nu}$.
Set  $\Phi = \Delta^k(\lambda\nu)$. By Corollary \ref{c expand downpath} and the description \eqref{ec flm k schur} of the strong Pieri operators,
\begin{align}
\label{e downpath case 2}
\fs^{(k)}_{\lambda\nu} \cdot u_{n_p} = H(\Phi ; \lambda\nu-\epsilon_{n_p} )
= \sum_{\{z \in \downpath_\Phi(n_p) \, \mid \, \lambda \nu - \epsilon_z \in \Par^k_\ell\}}
\!\!\!\! t^{B_\Phi(n_p,z)}\fs^{(k)}_{\lambda \nu - \epsilon_z}\, .
\end{align}
The hypotheses of the corollary hold by \eqref{et hyp 4} and the fact $n_p = p = r$ implies $\lambda_r  > \mu_r \ge \nu_1$.

Finally, applying the operator $\sum_{T \in \CTAB_{\alpha^-}(\mathbf{n})} u_{\creading(T)}$ to \eqref{e downpath case 2} yields
\begin{align}
\fs^{(k)}_{\lambda \nu} \cdot u_{n_p} \!\! \sum_{T \in \CTAB_{\alpha^-}(\mathbf{n})} \!\!\! u_{\creading(T)}
= \sum_{\{z \in \downpath_\Phi(n_p) \, \mid \, \lambda \nu - \epsilon_z \in \Par^k_\ell\}}
\!\!\!\!\!\! t^{B_\Phi(n_p,z)}\fs^{(k)}_{\lambda \nu - \epsilon_z} \cdot
\sum_{T \in \CTAB_{\alpha^-}(\mathbf{n})} \!\!\! u_{\creading(T)}\,. \notag
\end{align}
The left side agrees with the right side of \eqref{et schur kschur flag} by \eqref{e colwords split}, and
the right side agrees with the right side of \eqref{e case 2 LHS final} since $n_p = p$ and  $\zeta_p = \lambda_p$ implies  $\downpath_\Phi(p) = \downpath_\Psi(p)$.
This completes the proof in case 1.

\noindent \textbf{Step 3 (case 2): assemble terms from step 2.}

In this case ($\lambda^- = \lambda$), combining \eqref{e schur kschur two terms}, \eqref{e case 2 LHS mu plus}, and \eqref{e case 3 LHS nonroot} yields
\begin{align}
H(\Psi; \mu\nu) &=
\fs^{(k)}_{\lambda \nu} \cdot \!\! \sum_{T \in \CTAB_\alpha(\mathbf{n}^+)} \!\! u_{\creading(T)} \notag\\
&+ \sum_{\{z \in \downpath_\Psi(p)\setminus\{p\}  \, \mid\, (\mu\nu)_z > (\mu\nu)_{z+1} \}}
\!\!\!\! t^{B_\Psi(p,z)}  \Bigg( \fs^{(k)}_{\lambda \nu - \epsilon_z}  \cdot
\sum_{T \in \CTAB_{\alpha^-}(\mathbf{n})}  \!\! u_{\creading(T)}\Bigg) \,. \label{e case 3 LHS final}
\end{align}

\noindent \textbf{Step 4 (case 2): evaluate the right side of \eqref{et schur kschur flag} to match it with \eqref{e case 3 LHS final}.}

Since $\mathbf{n}^+$ is weakly increasing (Proposition \ref{p schur kschur flag helper} (e)),
we can use Proposition~\ref{p flag decomposition tableau} to obtain
\begin{align}
\label{e case 3 colwords split}
\sum_{T \in \CTAB_\alpha(\mathbf{n})} u_{\creading(T)} =
 \sum_{T \in \CTAB_{\alpha}(\mathbf{n}^+)} u_{\creading(T)} +
u_{n_p}\sum_{T \in \CTAB_{\alpha^-}(\mathbf{n})} u_{\creading(T)}\, .
\end{align}

We now compute the action of the operator in \eqref{e case 3 colwords split} on $\fs^{(k)}_{\lambda\nu}$.
Set  $\Phi = \Delta^k(\lambda\nu)$.
By Corollary \ref{c expand downpath} (the hypotheses hold by \eqref{et hyp 4} and $n_p < n^+_p \le r$),
\begin{align}
\label{e downpath case 3}
\fs^{(k)}_{\lambda\nu} \cdot u_{n_p} &= H(\Phi ; \lambda\nu-\epsilon_{n_p})
= \sum_{\{z \in \downpath_\Phi(n_p) \, \mid \, \lambda \nu - \epsilon_z \in \Par^k_\ell \}}
t^{B_\Phi(n_p,z)}\fs^{(k)}_{\lambda \nu - \epsilon_z}.
\end{align}

Using \eqref{e case 3 colwords split} followed by \eqref{e downpath case 3}, we obtain
\begin{align*}
&\fs^{(k)}_{\lambda \nu} \cdot \sum_{T \in \CTAB_{\alpha}(\mathbf{n})} u_{\creading(T)}
\, = \, \fs^{(k)}_{\lambda \nu} \cdot \bigg(\sum_{T \in \CTAB_{\alpha}(\mathbf{n}^+)} u_{\creading(T)} + u_{n_p} \sum_{T \in \CTAB_{\alpha^-}(\mathbf{n})} u_{\creading(T)} \bigg)\\
=\ & \fs^{(k)}_{\lambda \nu} \cdot \!\!\! \sum_{T \in \CTAB_{\alpha}(\mathbf{n}^+)} \!\! u_{\creading(T)} + \!
\sum_{\{z \in \downpath_\Phi(n_p) \, \mid \, \lambda \nu - \epsilon_z \in \Par^k_\ell \}} \!\!\!\!\!
t^{B_\Phi(n_p,z)}\fs^{(k)}_{\lambda \nu - \epsilon_z} \cdot \sum_{T \in \CTAB_{\alpha^-}(\mathbf{n})} u_{\creading(T)} \, .
\end{align*}
This agrees with the right side of \eqref{e case 3 LHS final} by Proposition \ref{p schur kschur flag helper} (g)
and because there is no $z=n_p$ term in the second sum; the latter follows from $\lambda_{n_p} = \lambda_{n_{p}+1}$, which comes from
Proposition \ref{p schur kschur flag helper} (f) and $n_p < n^+_p \le p$.
This completes the proof in case~2.
\end{proof}

\begin{example}
\label{ex proof schur kschur}
We illustrate case 1B and case 2 from the proof of Theorem \ref{t schur kschur flag}.
To keep the examples from getting too large, we have not appended zeros to  $\nu$ to force $\ell > k+r$ as
is done in Step 0.  However, we can match the examples below exactly with the proof
by appending  $k+r+1 - \ell$ zeros to the weights of all Catalan functions and modifying all root ideals as in the first paragraph of
Step 0.

\textbf{Example for case 1B.}
Let $k=8$, $r=4$, $\mu = 3212$, $\nu = 22221$, and $\Psi$ be the root ideal on the left in \eqref{eex case 1B}.
Hence $\zeta = 4543$, $\lambda = 4443$, $\alpha = 1231$, $\mathbf{n} = 1134$, and
\[
\CTAB_\alpha(\mathbf{n}) =
\fontsize{6pt}{5pt}\selectfont
\Bigg\{ \
\tableau{
\bl &\bl   & 1\\
\bl & 2 & 2\\
3& 3 & 3\\
\bl &\bl   & 4
}
\quad \ \ \tableau{
\bl &\bl   & 1\\
\bl & 1 & 2\\
3& 3 & 3\\
\bl &\bl   & 4
}
\ \Bigg\}.
\]
According to Theorem \ref{t schur kschur flag},
\begin{align}
\label{eex case1B terms}
H(\Psi;\mu\nu) = \fs^{(k)}_{\lambda \nu} \cdot \sum_{T \in \CTAB_{\alpha}(\mathbf{n})} u_{\creading(T)} = \fs^{(8)}_{444322221} \cdot \big(u_{3324321} + u_{3314321} \big).
\end{align}

Tracing through the proof of Theorem \ref{t schur kschur flag} for this example,
we have $p = 3$, $\lambda^- = 4433$, and so case 1B applies.
In the present example, \eqref{e schur kschur two terms} becomes
\ytableausetup{mathmode, boxsize=1.03em,centertableaux}
\begin{align}
\label{eex case 1B}
H(\Psi;\mu\nu) =
{\tiny
\begin{ytableau}
3& & & &      &*(red)&*(red)&*(red)&*(red)\\
 &2& & &      &*(red)&*(red)&*(red)  &*(red)\\
 & &1& &      &      &      & *(red) & *(red)\\
 & & &2&      &      &      &       &   \\
 & & & &2     &      &      &       &   \\
 & & & &      &2     &      &       &   \\
 & & & &      &     & 2     &       &   \\
 & & & &      &     &       & 2     &   \\
 & & & &      &     &       &       & 1  \\
\end{ytableau} =
\begin{ytableau}
3& & & &      &*(red)&*(red)&*(red)&*(red)\\
 &2& & &      &*(red)&*(red)&*(red)  &*(red)\\
 & &1& &      &      &      &       & *(red)\\
 & & &2&      &      &      &       &   \\
 & & & &2     &      &      &       &   \\
 & & & &      &2     &      &       &   \\
 & & & &      &     & 2     &       &   \\
 & & & &      &     &       & 2     &   \\
 & & & &      &     &       &       & 1  \\
\end{ytableau}
+ t\,
\begin{ytableau}
3& & & &      &*(red)&*(red)&*(red)&*(red)\\
 &2& & &      &*(red)&*(red)&*(red)  &*(red)\\
 & &2& &      &      &      & *(red) & *(red)\\
 & & &2&      &      &      &       &   \\
 & & & &2     &      &      &       &   \\
 & & & &      &2     &      &       &   \\
 & & & &      &     & 2     &       &   \\
 & & & &      &     &       & 1     &   \\
 & & & &      &     &       &       & 1  \\
\end{ytableau}\, .
}
\end{align}
For the first term on the right, the inductive hypothesis gives (as in \eqref{e case 2 LHS nonroot})
\begin{align}
\label{eex case1B term1}
{\tiny \begin{ytableau}
3& & & &      &*(red)&*(red)&*(red)&*(red)\\
 &2& & &      &*(red)&*(red)&*(red)  &*(red)\\
 & &1& &      &      &      &       & *(red)\\
 & & &2&      &      &      &       &   \\
 & & & &2     &      &      &       &   \\
 & & & &      &2     &      &       &   \\
 & & & &      &     & 2     &       &   \\
 & & & &      &     &       & 2     &   \\
 & & & &      &     &       &       & 1  \\
\end{ytableau}}
= \fs^{(k)}_{\lambda^-\nu} \cdot \sum_{T \in \CTAB_{\alpha^-}(\mathbf{n})} u_{\creading(T)} = \fs^{(8)}_{443322221} \cdot \big( u_{324321} + u_{314321} \big),
\end{align}
where $\alpha^- = 1221$ and
$\CTAB_{\alpha^-}(\mathbf{n}) = \CTAB_{1221}(1134) =
\fontsize{6pt}{5pt}\selectfont
\Bigg\{ \
\tableau{
\bl   & 1\\
2 & 2\\
3 & 3\\
\bl   & 4
}
\quad \ \ \tableau{
\bl   & 1\\
1 & 2\\
3 & 3\\
\bl   & 4
}
\ \Bigg\}.$

For the second term on the right of \eqref{eex case 1B}, the inductive hypothesis yields (as in \eqref{e case 2 LHS mu plus})
\begin{align}
\label{eex case1B term2}
{\tiny \begin{ytableau}
3& & & &      &*(red)&*(red)&*(red)&*(red)\\
 &2& & &      &*(red)&*(red)&*(red)  &*(red)\\
 & &2& &      &      &      & *(red) & *(red)\\
 & & &2&      &      &      &       &   \\
 & & & &2     &      &      &       &   \\
 & & & &      &2     &      &       &   \\
 & & & &      &     & 2     &       &   \\
 & & & &      &     &       & 1     &   \\
 & & & &      &     &       &       & 1  \\
\end{ytableau}}
= \fs^{(k)}_{\lambda \nu - \epsilon_8} \cdot \sum_{T \in \CTAB_{\alpha^-}(\mathbf{n})} u_{\creading(T)} = \fs^{(8)}_{444322211} \cdot \big( u_{324321} + u_{314321} \big).
\end{align}
The right side of \eqref{eex case1B term1} plus $t$ times the right side of \eqref{eex case1B term2} agrees with the right side of
\eqref{eex case1B terms} as is shown in Step 4 (case 1) of the proof of Theorem \ref{t schur kschur flag}.
\vspace{2mm}

\textbf{Example for case 2.}
Let $k = 8$, $r=4$, $\mu = 5442$, $\nu = 211$, and $\Psi$ be the root ideal on the left in \eqref{eex case 2}.
Hence $\zeta = 5666$, $\lambda = 5555$, $\alpha = 0113$,  $\mathbf{n} = 1123$, and
\[
\CTAB_\alpha(\mathbf{n}) =
\fontsize{6pt}{5pt}\selectfont
\Bigg\{\
\tableau{
\bl&\bl& \fr[r] \\
\bl&\bl& 1 \\
\bl&\bl& 2 \\
3&3& 3
}
\quad \tableau{
\bl&\bl& \fr[r] \\
\bl&\bl& 1 \\
\bl&\bl& 2 \\
3&3& 4
}
\quad \tableau{
\bl&\bl& \fr[r] \\
\bl&\bl& 1 \\
\bl&\bl& 3 \\
3&3& 4
}
\quad \tableau{
\bl&\bl& \fr[r] \\
\bl&\bl& 2 \\
\bl&\bl& 3 \\
3&3& 4
}
\quad \tableau{
\bl&\bl& \fr[r] \\
\bl&\bl& 1 \\
\bl&\bl& 2 \\
3&4& 4
}
\quad \tableau{
\bl&\bl& \fr[r] \\
\bl&\bl& 1 \\
\bl&\bl& 3 \\
3&4& 4
}
\quad \tableau{
\bl&\bl& \fr[r] \\
\bl&\bl& 2 \\
\bl&\bl& 3 \\
3&4& 4
}
\quad \tableau{
\bl&\bl& \fr[r] \\
\bl&\bl& 1 \\
\bl&\bl& 2 \\
4&4& 4
}
\quad \tableau{
\bl&\bl& \fr[r] \\
\bl&\bl& 1 \\
\bl&\bl& 3 \\
4&4& 4
}
\quad \tableau{
\bl&\bl& \fr[r] \\
\bl&\bl& 2 \\
\bl&\bl& 3 \\
4&4& 4
}\
\Bigg\}.
\]
According to Theorem \ref{t schur kschur flag},
\begin{align}
\label{eex case2 terms}
H(\Psi;\mu\nu) = 
\fs^{(8)}_{5555211} \cdot \bigg( {\small \parbox{7cm}{$u_{33321} + u_{33421} + u_{33431} + u_{33432} +u_{34421}$   $+u_{34431} + u_{34432} + u_{44421} + u_{44431} + u_{44432}$}} \bigg).
\end{align}

Tracing through the proof of Theorem \ref{t schur kschur flag} for this example,
we have $p = 4$, $\lambda^- = \lambda = 5555$ and so case 2 applies.
In the present example, \eqref{e schur kschur two terms} becomes
\ytableausetup{mathmode, boxsize=1.03em,centertableaux}
\begin{align}
\label{eex case 2}
{\tiny
H(\Psi;\mu\nu) =
\begin{ytableau}
5& & & &*(red)&*(red)&*(red)\\
 &4& & &*(red)&*(red)&*(red)\\
 & &4& &      &*(red)&*(red)\\
 & & &2&      &      &*(red)\\
 & & & &2     &       &\\
 & & & &      &1      & \\
 & & & &      &     & 1
\end{ytableau} =
\begin{ytableau}
5& & & &*(red)&*(red)&*(red)\\
 &4& & &*(red)&*(red)&*(red)\\
 & &4& &      &*(red)&*(red)\\
 & & &2&      &      &  \\
 & & & &2     &       &\\
 & & & &      &1      & \\
 & & & &      &     & 1
\end{ytableau}
+
t\,
\begin{ytableau}
5& & & &*(red)&*(red)&*(red)\\
 &4& & &*(red)&*(red)&*(red)\\
 & &4& &      &*(red)&*(red)\\
 & & &3&      &      &*(red)\\
 & & & &2     &       &\\
 & & & &      &1      & \\
 & & & &      &     & 0
\end{ytableau} \,.
}
\end{align}
For the first term on the right, the inductive hypothesis gives (as in \eqref{e case 3 LHS nonroot})
\begin{align}
\label{eex case2 term1}
{\tiny
\begin{ytableau}
5& & & &*(red)&*(red)&*(red)\\
 &4& & &*(red)&*(red)&*(red)\\
 & &4& &      &*(red)&*(red)\\
 & & &2&      &      &  \\
 & & & &2     &       &\\
 & & & &      &1      & \\
 & & & &      &     & 1
\end{ytableau}
}
= \fs^{(k)}_{\lambda \nu} \cdot \sum_{T \in \CTAB_{\alpha}(\mathbf{n}^+)} u_{\creading(T)}
= \fs^{(8)}_{5555211} \cdot \big( u_{44421} + u_{44431} + u_{44432}\big),
\end{align}
where  $\mathbf{n}^+ = \mathbf{n}+\epsilon_p = 1124$ and
$\CTAB_\alpha(\mathbf{n}^+) = \CTAB_{0113}(1124) =
\fontsize{6pt}{5pt}\selectfont
\Bigg\{ \
\tableau{
\bl&\bl& \fr[r] \\
\bl&\bl& 1 \\
\bl&\bl& 2 \\
4&4& 4
}
\quad \ \ \tableau{
\bl&\bl& \fr[r] \\
\bl&\bl& 1 \\
\bl&\bl& 3 \\
4&4& 4
}
\quad \ \ \tableau{
\bl&\bl& \fr[r] \\
\bl&\bl& 2 \\
\bl&\bl& 3 \\
4&4& 4
}
\ \Bigg\}.$

For the second term on the right of \eqref{eex case 2}, the inductive hypothesis yields (as in \eqref{e case 2 LHS mu plus})
\begin{multline}
\label{eex case2 term2}
{\tiny
\begin{ytableau}
5& & & &*(red)&*(red)&*(red)\\
 &4& & &*(red)&*(red)&*(red)\\
 & &4& &      &*(red)&*(red)\\
 & & &3&      &      &*(red)\\
 & & & &2     &       &\\
 & & & &      &1      & \\
 & & & &      &     & 0
\end{ytableau}
}
= \fs^{(k)}_{\lambda \nu-\epsilon_7} \cdot \sum_{T \in \CTAB_{\alpha^-}(\mathbf{n})} u_{\creading(T)} \\
= \fs^{(8)}_{5555210} \cdot \big(u_{3321} + u_{3421} + u_{3431} + u_{3432} + u_{4421} + u_{4431} + u_{4432} \big),
\end{multline}
where  $\alpha^- = \alpha-\epsilon_p = 0112$ and
\[
\CTAB_{\alpha^-}(\mathbf{n}) = \CTAB_{0112}(1123) =
\fontsize{6pt}{5pt}\selectfont
\Bigg\{ \
\tableau{
\bl& \fr[r] \\
\bl& 1 \\
\bl& 2 \\
3& 3
}
\quad \ \ \tableau{
\bl& \fr[r] \\
\bl& 1 \\
\bl& 2 \\
3& 4
}
\quad \ \ \tableau{
\bl& \fr[r] \\
\bl& 1 \\
\bl& 3 \\
3& 4
}
\quad \ \ \tableau{
\bl& \fr[r] \\
\bl& 2 \\
\bl& 3 \\
3& 4
}
\quad \ \ \tableau{
\bl& \fr[r] \\
\bl& 1 \\
\bl& 2 \\
4& 4
}
\quad \ \ \tableau{
\bl& \fr[r] \\
\bl& 1 \\
\bl& 3 \\
4& 4
}
\quad \ \ \tableau{
\bl& \fr[r] \\
\bl& 2 \\
\bl& 3 \\
4& 4
}
\ \Bigg\}.
\]
The right side of \eqref{eex case2 term1} plus  $t$ times the right side of \eqref{eex case2 term2} agrees with
the right side of
\eqref{eex case2 terms} as is shown in Step 4 (case 2) of the proof of Theorem \ref{t schur kschur flag}.
\end{example}

%

\section{Proof of Theorems \ref{t k split to k schur} and \ref{t k split eq catalan}}
\label{s ksplit proofs}

We give the proofs of Theorems \ref{t k split to k schur} and \ref{t k split eq catalan} after three lemmas.
The proof of Theorem~\ref{t k split to k schur} is by an induction on the number of partitions in the  $k$-split using Theorem~\ref{t schur kschur}, and
Theorem~\ref{t k split eq catalan}  follows easily from
Theorem \ref{t k split to k schur} and its proof.

Our first lemma is related to the  $k$-rectangle property (Corollary \ref{c k rectangle}).  Note that
we do not yet know that the   $k$-Schur Catalan functions have the  $k$-rectangle property since Corollary~\ref{c k rectangle} relies on what we are now proving.

\begin{lemma}
\label{l multiply by rectangle easy}
Let $U = (k-r+1)^r$ be a $k$-rectangle and $(\Psi, \gamma)$ an indexed root ideal of length  $\ell-r$ such that
 $\nr(\Psi)_1 \ge \min(r-1, \ell-r-1)$.
Let $\tilde{\Psi} \subset \Delta^+_\ell$ be the root ideal defined by $\nr(\tilde{\Psi})_i = \min(r-1,\ell-i)$ for  $i \le r$ and
$\nr(\tilde{\Psi})_i = \nr(\Psi)_{i-r}$ for  $i > r$.
Then
\begin{align}
\bb_U \mysquareb H(\Psi; \gamma) = H(\tilde{\Psi}; U\gamma).
\end{align}
In particular, if  $\nu \in \Par^{k-r+1}_{\ell-r}$, then $\bb_U \mysquareb \fs^{(k)}_\nu = \fs^{(k)}_{U\nu}$.
\end{lemma}
\begin{proof}
By Proposition \ref{p Catalan box times} and \eqref{e Catalan op eq fun}, $\bb_U \mysquareb H(\Psi; \gamma) =
H(\varnothing_r \uplus \Psi;U \gamma)$.
Then applying Lemma~\ref{l cascading toggle lemma} with $z = r-1, r-2, \dots, 1$ shows that
$H(\varnothing_r \uplus \Psi; U \gamma) = H(\Phi;U \gamma)$, where the root ideal  $\Phi \subset \Delta^+_\ell$ is obtained from  $\varnothing_r \uplus \Psi$ by removing one root in rows  $2,\dots, r$.  We can again modify  $\Phi$ using Lemma~\ref{l cascading toggle lemma} with $z = r-1, r-2, \dots, 2$ to remove one root in rows  $3,\dots, r$.
Continuing in this way, we obtain the root ideal  $\tilde{\Psi}$.
\end{proof}

Throughout this section we will work with the following  free $\ZZ[t]$-submodules of $\Lambda^k$\,:
\begin{align*}
\Omega^{k,a} &:= \Span_{\ZZ[t]} \{ \fs^{(k)}_\mu \mid \mu\in \Par^k, \, \mu_1 = a \}, \\
\Lambda^{k,d}&:= \textstyle \bigoplus_{a \le d}  \Omega^{k,a} =  \Span_{\ZZ[t]} \{ \fs^{(k)}_\mu \mid \mu \in \Par^{d}\}.
\end{align*}

The next result shows that the strong Pieri operators commute with certain generalized Hall-Littlewood vertex operators $\bb_U$, up to a shift in indices.
Be aware that although the strong Pieri operators act on the
right and the generalized Hall-Littlewood vertex operators act on the left, they do not commute,
so we must take care with parentheses.

\begin{lemma}
\label{l commute with square}
Let $U = (k-r+1)^r$ be a $k$-rectangle and $\nu \in \Par^{k-r+1}_{\ell-r}$ so that  $U\nu\in \Par^k_\ell$.
Let $w = w_1 \cdots w_d$ be a word in letters $[\ell-r]$ and let $w^+  = (w_1+r) \cdots (w_{d}+r)$.
Suppose
\begin{align}
\label{e commute with square hyp}
\text{$\fs^{(k)}_\nu \cdot u_{w_1} \cdots u_{w_j} \in \Lambda^{k,k-r+1}$ \quad for all  $j \in [d]$.}
\end{align}
Then
\begin{align}
\label{el commute with square}
\bb_U \mysquare \left(\fs^{(k)}_\nu \cdot \myhat{u}_{w}\right) =
\left(\bb_U \mysquareb \fs^{(k)}_\nu \right) \cdot u_{w^+} = \fs^{(k)}_{U\nu} \cdot u_{w^+}.
\end{align}
Additionally,
$\fs^{(k)}_{U\nu} \cdot u_{w^+} \in \Span_{\ZZ[t]} \! \big\{\fs^{(k)}_{U\beta} \mid \beta \in \Par^{k-r+1} \big\}$.
\end{lemma}
\begin{proof}
The proof of \eqref{el commute with square} is by induction on  $d$.
The second equality holds by Lemma~\ref{l multiply by rectangle easy};  this given, the base case  $d=0$
is trivial and it remains to prove that for  $d > 0$
the left and right sides of \eqref{el commute with square} agree.
Set $v = w_1 \cdots w_{d-1}$, $v^+ = (w_1+r) \cdots (w_{d-1}+r)$, and  $p =w_d$.
By \eqref{e commute with square hyp}, we can write
$\fs^{(k)}_\nu \cdot \myhat{u}_{v}  =
\sum_{\eta \in \Par^{k-r+1}_{\ell-r}} c_\eta \, \fs^{(k)}_\eta$ for  $c_\eta  \in \NN[t]$,
where we have used $\fs^{(k)}_{\mu}= \fs^{(k)}_{(\mu,0)} = \fs^{(k)}_{(\mu,0,0)} = \cdots$ to write the sum over
$\Par^{k-r+1}_{\ell-r}$ rather than  $\Par^{k-r+1}$.
Then by the inductive hypothesis and Lemma~\ref{l multiply by rectangle easy},
\begin{align}
\label{e ind hyp}
\fs^{(k)}_{U \nu} \cdot u_{v^+} =
\bb_U \mysquare (\fs^{(k)}_\nu \cdot \myhat{u}_{v})
= \bb_U \mysquare \bigg( \sum_{\eta \in \Par^{k-r+1}_{\ell-r}} \!\!\! c_\eta \, \fs^{(k)}_\eta \bigg)
 = \sum_{\eta \in \Par^{k-r+1}_{\ell-r}} \!\!\! c_\eta \, \fs^{(k)}_{U \eta}.
\end{align}
We compute
\begin{align*}
\bb_U \mysquare (\fs^{(k)}_\nu \cdot \myhat{u}_{v} \myhat{u}_{p}) =&
\sum_{\eta \in \Par^{k-r+1}_{\ell-r}} \!\!\! c_\eta \, \bb_U \mysquare (\fs^{(k)}_\eta  \cdot \myhat{u}_{p}) =
\sum_{\eta \in \Par^{k-r+1}_{\ell-r}} \!\!\! c_\eta \, \bb_U \mysquareb H(\Delta^k(\eta); \eta- \epsilon_{p})
 \\
=& \sum_{\eta \in \Par^{k-r+1}_{\ell-r}} \!\!\! c_\eta \,  H(\Delta^k(U\eta); U\eta - \epsilon_{p+r})  =
\sum_{\eta \in \Par^{k-r+1}_{\ell-r}} \!\!\! c_\eta \, \fs^{(k)}_{U \eta} \cdot u_{p+r} =
\fs^{(k)}_{U \nu} \cdot u_{w^+}.
\end{align*}
The second and fourth equalities are by the description \eqref{ec flm k schur} of the strong Pieri operators,
the third equality is by Lemma~\ref{l multiply by rectangle easy}, and the fifth is by \eqref{e ind hyp}.

For the second statement, \eqref{el commute with square}, \eqref{e commute with square hyp},  and Lemma~\ref{l multiply by rectangle easy} yield
\[
\fs^{(k)}_{U\nu} \cdot u_{w^+}
= \bb_U \mysquare \left(\fs^{(k)}_\nu \cdot \myhat{u}_{w}\right)
\in \bb_U \mysquare \Lambda^{k, k-r+1} \subset
 \Span_{\ZZ[t]}\! \big\{\fs^{(k)}_{U\beta} \mid \beta \in \Par^{k-r+1} \big\}.
 \qedhere
\]
\end{proof}

It turns out that under a mild assumption, the strong Pieri operators in Theorem \ref{t schur kschur} do not change the first part of the partition:

\begin{lemma}
\label{c height 1 ribbon}
As in  Theorem \ref{t schur kschur}, let
$\mu \in \Par^{k-r+1}_r$ and $\nu \in \Par^k$ be such that $\mu\nu$ is a partition;
set $U = (k-r+1)^r$.
Suppose the first row of  $U/\mu$ is empty, i.e., $\mu_1 = U_1$.
Let  $T \in \SSYT_{U/\mu}([r])$ and $w = w_1 \cdots w_{|U/\mu|} = \creading(T)$.
Then for any  $j \le |U/\mu|$,
\begin{align}
\label{et height 1 ribbon}
\fs^{(k)}_{U\nu} \cdot u_{w_1} \cdots u_{w_j} \, \in \Omega^{k,k-r+1}.
\end{align}
\end{lemma}
\begin{proof}
Write $\fs^{(k)}_{U\nu} \cdot u_{w_1} \cdots u_{w_j} =  \sum_{\tau \in \Par^k} c_\tau \fs^{(k)}_{\tau} $ with coefficients $c_\tau \in \NN[t]$.
We apply Corollary \ref{c expand downpath} for each strong Pieri operator  $u_{w_1}, \dots, u_{w_j}$.
The assumptions of the corollary are satisfied: any
term in the $k$-Schur function expansion of  $\fs^{(k)}_{U\nu} \cdot u_{w_1} \cdots u_{w_i}$ for  $i \in [j]$ is of the form
$d_\eta\, \fs^{(k)}_\eta$ for  $\eta \subset U\nu$ and thus $\nr(\Delta^k(\eta))_i \ge \nr(\Delta^k(U\nu))_i \ge r-i$ for $i \in [r]$;
the second assumption $\big( p = r < \ell  \, \implies \, \eta_r > \eta_{r+1} \big)$ follows from $\mu_r \ge \nu_1$ and the fact that the number of $r$'s in the word $w$ is
at most $U_r-\mu_r$.
The corollary then tells us that
$c_\tau = 0$ unless there exists a saturated chain
$U\nu = \tau^0 \supset \tau^1 \supset \cdots \supset \tau^j = \tau$ in Young's lattice
with  $\tau^a = \tau^{a-1} - \epsilon_{b_a}$ for all  $a \in [j]$, where the integers $b_1, \dots, b_j$ satisfy $b_1 \ge w_1, \, \dots,\, b_j \ge w_j$.
It follows that if some $c_\tau \ne 0$ with  $\tau_1 \ne U_1= k-r+1$, then  $w_1 \cdots w_j$ contains
 $r \ r-1 \, \cdots \, 2 \ 1$ as a subsequence.
But this is impossible since the first row of  $U/\mu$ is empty and any strictly decreasing subsequence of  $w = \creading(T)$
uses at most one entry from each row of  $T$.
This proves that $\tau_1 = U_1 = k-r+1$ whenever $c_\tau \ne 0$, as desired.
\end{proof}

\begin{proof}[Proof of Theorem \ref{t k split to k schur}]
We prove by induction on $d$ the following stronger claim:
\break
in addition to \eqref{et k split to k schur},
for a $u$-monomial $u_w = u_{w_1} \cdots u_{w_m}$ arising in the expansion of
 $\Big( \sum_{T \in \SSYT_{\theta^d}(N_d)} \! u_{\creading(T)}\Big)
\cdots
\Big( \sum_{T \in \SSYT_{\theta^1}(N_1)} \! u_{\creading(T)}\Big)$
as a sum of  $u$-monomials, we have $\fs^{(k)}_{U} \cdot u_{w_1} \cdots u_{w_j} \in \Omega^{k, k-r_1+1} = \Omega^{k,\lambda_1}$ for all  $j \in [m]$.

It makes sense to take the base case to be  $d=0$ which occurs exactly when  $\lambda$ is the empty partition, and then
 $G^{(k)}_\emptyset = 1 = \fs^{(k)}_\emptyset$ is the desired result.
 Now assume  $d > 0$.

Since ($\lambda^2, \dots, \lambda^d$) is the  $k$-split of  $(\lambda_{r_1+1}, \lambda_{r_1+2}, \dots)$, the inductive hypothesis yields
\begin{align}
\label{e k split to k schur induct 0}
\!\! \bb_{\lambda^2}\mysquare  \cdots \mysquare \bb_{\lambda^{d-1}} \mysquareb s_{\lambda^d} = \fs^{(k)}_{\hat{U}}
\cdot \bigg( \sum_{T \in \SSYT_{\theta^d}(\hat{N}_d)} \!\!\!\! u_{\creading(T)}\bigg)
\cdots
\bigg( \sum_{T \in \SSYT_{\theta^2}(\hat{N}_2)} \!\!\!\! u_{\creading(T)}\bigg)
\end{align}
where $\hat{U}$ is  the partition $U$ restricted to rows $> r_1$ and  $\hat{N}_j := \{i-r_1\mid i \in N_j\}$ for \break $j = 2, \dots, d$.
Moreover, by the inductive hypothesis for the strengthened claim,
the right side of \eqref{e k split to k schur induct 0} is a sum of $k$-Schur functions $\sum_{\nu}c_{\nu} \, \fs^{(k)}_\nu$ over
$\big\{\nu \in \Par^k \mid \nu_1 = (\lambda^2)_1 \big\}$ provided  $d > 1$ (if  $d=1$, \eqref{e k split to k schur induct 0}
simply reads $G^{(k)}_\emptyset = \fs^{(k)}_\emptyset$).
Hence for each  $\nu$ appearing in this sum, Theorem \ref{t schur kschur} applies and yields
\begin{align}
\label{e k split k schur 2}
\bb_{\lambda^{1}} \mysquareb \fs^{(k)}_\nu = \fs^{(k)}_{U^1 \nu} \cdot \sum_{T \in \SSYT_{\theta^{1}}(N_1)} \!\! u_{\creading(T)}
 = \big(\bb_{U^1} \mysquareb \fs^{(k)}_{\nu} \big) \cdot \sum_{T \in \SSYT_{\theta^{1}}(N_1)} \!\! u_{\creading(T)} \, ,
\end{align}
where the second equality is by Lemma~\ref{l multiply by rectangle easy}.

Applying $\bb_{\lambda^1} \mysquare$ to both sides of \eqref{e k split to k schur induct 0} yields
\begin{align*}
&G^{(k)}_\lambda =  \bb_{\lambda^1} \mysquare \left( \fs^{(k)}_{\hat{U}}
\cdot \Bigg( \sum_{T \in \SSYT_{\theta^d}(\hat{N}_d)} \!\!\!\!\!\!\! u_{\creading(T)}\Bigg)
\cdots
\Bigg( \sum_{T \in \SSYT_{\theta^2}(\hat{N}_2)} \!\!\!\!\!\!\! u_{\creading(T)}\Bigg)\right)\\
&= \left( \bb_{U^1} \mysquare \left( \fs^{(k)}_{\hat{U}}
\cdot \Bigg( \sum_{T \in \SSYT_{\theta^d}(\hat{N}_d)} \!\!\!\!\!\!\!  u_{\creading(T)}\Bigg)
\cdots
\Bigg( \sum_{T \in \SSYT_{\theta^2}(\hat{N}_2)} \!\!\!\!\!\!\! u_{\creading(T)}\Bigg)\right) \right) \cdot \sum_{T \in \SSYT_{\theta^{1}}(N_1)} \!\!\!\!\!\!\! u_{\creading(T)}\\
&= \fs^{(k)}_{U} \cdot \Bigg( \sum_{T \in \SSYT_{\theta^d}(N_d)} \!\!\!\!\!\!\! u_{\creading(T)}\Bigg)
\cdots
\Bigg( \sum_{T \in \SSYT_{\theta^2}(N_2)} \!\!\!\!\!\!\! u_{\creading(T)}\Bigg)
\Bigg(\sum_{T \in \SSYT_{\theta^{1}}(N_1)} \!\!\!\!\!\!\! u_{\creading(T)}\Bigg),
\end{align*}
where the second equality is by \eqref{e k split k schur 2} and the third is by Lemma~\ref{l commute with square} (the assumption \eqref{e commute with square hyp} holds by the inductive hypothesis for the strengthened claim).
This proves \eqref{et k split to k schur}.

To complete the proof of the strengthened claim, let  $u_w = u_{w_1} \cdots u_{w_j}$,  $j \in [m]$, be as in the first paragraph of the proof.
Set  $m' = \min(|\theta^d| + \dots + |\theta^2|, j)$.
Write  $w_1\cdots w_j = v v'$, where  $v = w_1 \cdots w_{m'}$ and  $v' = w_{m'+1} \cdots w_{j}$ ($v$ or  $v'$ may be empty).
By the second statement in Lemma~\ref{l commute with square} applied with  $r = r_1$, $U =U^1$,  $\nu = \hat{U}$, and  $w^+ = v$ (again, the assumption \eqref{e commute with square hyp} holds by the inductive hypothesis for the
strengthened claim),
\[
\fs^{(k)}_{U} \cdot u_{v}  = \sum_{\eta \in \Par^{k - r_1 + 1}} d_\eta \, \fs^{(k)}_{U^1 \eta} \text{\qquad for  $d_\eta \in \NN[t]$}.
\]
Lemma~\ref{c height 1 ribbon} then yields (the first row of $\theta^1 = U^1/ \lambda^1$ is empty since  $k-r_1+1 = \lambda_1$)
\[
\fs^{(k)}_{U} \cdot u_{w_1} \cdots u_{w_{j}}= \fs^{(k)}_{U} \cdot u_{v} u_{v'}
  = \bigg( \sum_{\eta \in \Par^{k-r_1+1}} d_\eta \, \fs^{(k)}_{U^1 \eta} \bigg) \cdot u_{v'} \in \Omega^{k,k-r_1+1}.
\qedhere
\]
\end{proof}

\begin{proof}[Proof of Theorem \ref{t k split eq catalan}]
We prove $\tilde{A}^{(k)}_\mu = \fs^{(k)}_{\mu}$ by induction on $\ell(\mu)$.
The base case  $\ell(\mu) = 0$ holds since  $\tilde{A}^{(k)}_\emptyset = 1 = \fs^{(k)}_\emptyset$.
Now assume  $\ell(\mu) > 0$ and
set $\hat{\mu} := (\mu_2, \mu_3, \dots, \mu_\ell)$.
By the inductive hypothesis,  $\tilde{A}^{(k)}_{\hat{\mu}} =\fs^{(k)}_{\hat{\mu}}$.
We evaluate definition \eqref{ed k split} by first computing $\bb_{\mu_1} \mysquareb \tilde{A}^{(k)}_{\hat{\mu}} = \bb_{\mu_1} \mysquareb \fs^{(k)}_{\hat{\mu}}$.
Corollary~\ref{c Br on k schur} yields the first equality below:
\begin{align}
\label{e expand Br}
\bb_{\mu_1} \mysquareb \fs^{(k)}_{\hat{\mu}}
= \fs^{(k)}_{(k,\hat{\mu})} \cdot u_1^{k-\mu_1}
= \fs^{(k)}_\mu + \sum_{\{\nu \in \Par^k \, \mid \, \nu_1 > \mu_1\}} c_{\nu} \, \fs^{(k)}_\nu,  \quad \text{ for coefficients $c_\nu \in \NN[t]$.}
\end{align}
The second equality follows from applying Corollary \ref{c expand downpath}  $k-\mu_1$ times with  $r=1$, noting that the assumption
$\eta_r > \eta_{r+1}$ holds by $\mu_1 \ge \mu_2$.

By the strengthened claim in the proof of Theorem \ref{t k split to k schur},  $G^{(k)}_\lambda \in \Omega^{k,\lambda_1}$
for all  $\lambda \in \Par^k$. It follows that
\begin{align*}
& \tilde{\Omega}^{k,a} = \, \Span_{\ZZ[t]} \! \big\{G^{(k)}_\lambda \mid \lambda \in \Par^k, \lambda_1 = a\big\}
 \ = \ \Span_{\ZZ[t]} \! \big\{\fs^{(k)}_\lambda \mid \lambda \in \Par^k, \lambda_1 = a\big\} = \Omega^{k,a}
\end{align*}
for all $a \in [0,k]$.
Hence applying the projection  $\pi^{k,d}$ from definition \eqref{ed k split} to the right side of
\eqref{e expand Br} sends the term $\fs^{(k)}_\mu$ to itself and sends each term $c_\nu \, \fs^{(k)}_\nu$ in the sum to 0, giving
\[\tilde{A}^{(k)}_{\mu} = \pi^{k,\mu_1}(\bb_{\mu_1} \mysquareb \tilde{A}^{(k)}_{\hat{\mu}} ) = \pi^{k,\mu_1}(\bb_{\mu_1} \mysquareb \fs^{(k)}_{\hat{\mu}} ) = \fs^{(k)}_\mu. \qedhere \]
\end{proof}

\vspace{2mm}
\noindent
\textbf{Acknowledgments.}
We are grateful to Mark Shimozono for helpful conversations and
to Elaine~So for help typing and typesetting figures.

\bibliographystyle{plain}
\bibliography{mycitations}
\def\cprime{$'$} \def\cprime{$'$} \def\cprime{$'$}

\end{document}